\documentclass[11pt,a4paper]{article}

\usepackage{amsfonts}
\usepackage{amsmath}
\usepackage{amsbsy}
\usepackage{theorem}
\usepackage{graphicx}

\newtheorem{theorem}{Theorem}

\newtheorem{lemma}{Lemma}

\newtheorem{remark}{Remark}

\begin{document}
\title{Macroscale structural complexity analysis of subordinated spatiotemporal random fields}
\author{José M. Angulo  and María D. Ruiz-Medina}
\date{}

 \maketitle

\bigskip

\begin{abstract} Large-scale behavior of a wide class of spatial and spatiotemporal processes is characterized in terms of informational measures. Specifically, subordinated random fields defined by non-linear transformations on the family of homogeneous and isotropic Lancaster-Sarmanov random fields are studied under long-range dependence (LRD) assumptions. In the spatial case, it is shown that Shannon mutual information beween marginal distributions for infinitely increasing distance, which can be properly interpreted as a measure of macroscale structural complexity and diversity, has an asymptotic power decay that directly depends on the underlying LRD parameter, scaled by the subordinating function rank. Sensitivity with respect to distortion induced by the deformation parameter under the generalized form given by divergence-based Rényi mutual information is also analyzed. In the spatiotemporal framework, a spatial  infinite-dimensional  random field approach is adopted. The study of the large-scale asymptotic behavior is then extended under the proposal of a functional formulation of the Lancaster-Sarmanov random field class, as well as of divergence-based mutual information. Results are illustrated, in the context of geometrical analysis of sample paths, considering some scenarios based on Gaussian and Chi-Square subordinated spatial and spatio-temporal random fields.
\end{abstract}




\section{Introduction}
\label{sec:intro}

A growing interest  is observed, in the last few decades, on the parametric  (Bosq, 2000) and nonparametric (Ferraty and Vieu, 2006) spatiotemporal data  analysis based on infinite-dimensional spatial models. Particularly, the field of Functional Data Analysis (FDA) has been nurtured by various disciplines, including probability in abstract spaces (see, e.g., Ledoux and Talagrand, 1991).  The most recent contributions are developed in the framework of stochastic partial differential and pseudodifferential equations (see, e.g., Ruiz-Medina, 2022). Particularly in the Gaussian random field context, it is well-known that Gaussian measures in Hilbert spaces and associated infinite-dimensional quadratic forms play a crucial role (see Da Prato and Zabczyk, 2002). The tools developed have recently  been  exploited in several statistical papers on  spatiotemporal modeling under the umbrella of infinite-dimensional inference, based on statistical analysis of infinite-dimensional multivariate data and stochastic processes (see, e.g., Frías et al., 2022; Ruiz-Medina, 2022; Torres-Signes et al., 2021).  In particular, the statistical distance based approach  constitutes a useful tool in hypothesis testing (see, e.g., Ruiz-Medina, 2022, where  an estimation methodology based on a Kullback-Leibler divergence-like loss operator is proposed in the temporal  functional  spectral   domain). This methodology has been also exploited in structural complexity analysis based on sojourn measures of spatiotemporal Gaussian and Gamma-correlated random fields. Indeed, there exists a vast literature in the context of stochastic  geometrical analysis of the sample paths of random fields based on these measures (see, e.g.,  Bulinski et al., 2012; Ivanov and Leonenko, 1989, among others).  Special attention has been paid to the asymptotic analysis of long-range dependent  random fields (see Leonenko, 1999; Leonenko and Olenko, 2014; Makogin and Spodarev, 2022). Recently, in Leonenko and Ruiz-Medina (2022), new spatiotemporal limit results have been derived to analyze, in particular, the limit distribution of Minkowski functionals, in the context of Gaussian and Chi-Square subordinated spatiotemporal random fields.

The geometrical interpretation of these functionals, which, for instance in 2D, is related to the total area of all \emph{hot} regions, and the  total length of the boundary between \emph{hot} and \emph{cold} regions, as well as the Euler-Poincaré characteristic, counting the number of isolated \emph{hot} regions minus the number of isolated \emph{cold} regions within the \emph{hot} regions, has motivated  several statistical approaches, adopted, for instance, in the  Cosmic Microwave Background (CMB) evolution modeling and data analysis (see, e.g., Marinucci and Peccati, 2011).

The present work continues the above-referred research lines in relation to  structural complexity analysis of long-range dependent Gaussian and Chi-Square subordinated spatial and spatiotemporal random fields. Indeed, a more general random field  framework is considered, defined from the Lancaster-Sarmanov random field class.

The approach adopted here is based on the quantitative assessment in terms of appropriate information-theoretic measures, a framework which has played a fundamental role, with a very extensive related literature, in the probabilistic and statistical characterization and description of structural aspects inherent to stochastic systems arising in a wide variety of knowledge areas. More precisely, the asymptotic behavior, for infinitely increasing distance, of divergence-based Shannon and R\'enyi mutual information measures, which are formally and conceptually connected to certain forms of `complexity' and `diversity' (see, for instance, Angulo et al. 2021, and references therein), is derived for this random field class. This behavior is characterized by the long-range dependence (LRD) parameter, that in this context determines the global diversity loss, associated with lower values of such a parameter.  The deformation parameter $q \neq 1$ involved in the definition of  R\'enyi mutual information also modulates the power rate of decay of this structural dependence indicator as the distance between the considered spatial location increases. As is well-known, the derived asymptotic analysis results based on Shannon mutual information arise as limiting cases, for $q \to 1$, of the ones obtained based on R\'enyi mutual information. The spatiotemporal case is analyzed in an infinite-dimensional spatial framework. In particular, related spatiotemporal extensions of the Lancaster-Sarmanov random field class, as well as of divergence-based mutual information measures, are formalized under a functional approach. A simulation study is undertaken showing, in particular, that the same asymptotic orders as in the purely spatial case hold for the infinite-dimensional versions of divergence measures here considered (see, e.g., Angulo and Ruiz-Medina, 2023).

The paper content is structured as follows. Section \ref{section:preliminaries} provides the preliminary elements on the analyzed class of Lancaster-Sarmanov random fields, as well as the introduction of information-complexity measures.  The main results derived on the asymptotic macroscale behavior of  Shannon and R\'enyi mutual information measures, involving the bivariate distributions of the Lancaster-Sarmanov subordinated random field class studied, are obtained in Section  \ref{section:methodology-MI}. These aspects are illustrated in terms of some numerical examples in subsection \ref{subsection:spatial-simulations}. The functional approach to the spatiotemporal case based on an infinite-dimensional spatial framework is addressed in Section \ref{section:st-case}. Final comments, with a reference to open research lines, are given in Section \ref{section:conclusion}.

\section{Preliminaries}
\label{section:preliminaries}

Let $(\Omega, \mathcal{A},P)$ be the basic complete probability space, and denote by \linebreak $\mathcal{L}^{2}(\Omega,\mathcal{A},P)$ the Hilbert space of equivalence classes (with respect to $P$) of zero-mean second-order random variables on  $(\Omega, \mathcal{A},P)$.
Consider  $X=\{ X(\mathbf{z}),\ \mathbf{z}\in \mathbb{R}^{d}\}$ to be a zero-mean spatial homogeneous and isotropic mean-square continuous second-order random field, with correlation function $\gamma
(\|\mathbf{x}-\mathbf{y}\|)=\mbox{Corr}\left(X(\mathbf{x}),X(\mathbf{y})\right)$. Assume that the marginal probability distributions are absolutely continuous, having probability density $p(u)$ with support included in $(a,b)$, $-\infty\leq a<b\leq \infty$.
Let now $L^{2}((a,b), p(u)du)$  be the Hilbert space of equivalence classes of measurable real-valued functions on the interval $(a,b)$ which are square-integrable with respect to the measure $\mu (du)=p(u)du$.

\subsection{Lancaster-Sarmanov random field class}
\label{subsection:LS-class}
Assume that  there exists a  complete  orthonormal basis $\{e_{k},\ k\geq 0\}$,  with $e_{0}\equiv 1$,  of the space
$L^{2}((a,b), p(u)du)$  such that
\begin{eqnarray}&&
\frac{\partial^{2}}{\partial
u\partial v}P\left[ X(\mathbf{x})\leq u,X(\mathbf{y})\leq
v\right]=:
p(u,v,\|\mathbf{x}-\mathbf{y}\|)\nonumber\\&& =p(u)~p(v)~\left[ 1+\sum_{k=1}^{\infty }\gamma^{k}
(\|\mathbf{x}-\mathbf{y}\|)~e_{k}(u)~e_{k}(v)\right], \quad \forall \mathbf{x}, \mathbf{y} \in \mathbb{R}^{d}  .
  \label{210bb}
\end{eqnarray}
The family  of random fields $X$ satisfying the above conditions is known as the \emph{Lancaster-Sarmanov} random field class (see, e.g.,
Lancaster, 1958; Leonenko et al., 2017;  Sarmanov, 1963).  Gaussian  and Gamma-correlated random fields are two important cases within this class, with  $\{e_{k},\ k\geq 0\}$ being given by the (normalized) Hermite polynomial system in the Gaussian case (see, for example, Peccati and Taqqu, 2011), and by the Laguerre polynomials in the Gamma-correlated case. An interesting special case of the latter is defined by the Chi-Square random field family.

It is well-known that non-linear transformations of these random fields can be approximated in terms of the above series expansions: For every $g\in L^{2}((a,b), p(u)du)$,
\begin{equation} \label{ex-0}
g(x)=\sum_{k=0}^{\infty}C_{k}^{g}e_{k}(x),\quad \mbox{with} \quad C_{k}^{g} =\int_a^b g(u)~e_{k}(u)~p(u)~du,\ k\geq 0.
\end{equation}
In particular, $C_{0}^{g} = E_p[g(X)]$. The minimum integer $m$ such that $C_{k} = 0$ for all $1 \leq k \leq m-1$ represents the \emph{rank} of function $g$ in the orthonormal basis $\{e_{k},\ k\geq 0\}$  of the space
$L^{2}((a,b), p(u)du)$; that is,
\begin{equation} \label{ex}
g(x)=C_{0}^{g}+\sum_{k=m}^{\infty}C_{k}^{g}e_{k}(x).
\end{equation}
In the cases of Gaussian and Gamma-correlated  subordinated random fields we will refer to the Hermite and Laguerre ranks, respectively, of function $g$.

An interesting example is Minkowski functional $M_{0}(\nu; X,D)$ providing the random volume of the set of spatial points within $D$ (usually a bounded subset of $\mathbb{R}^{d}$) where random field $X$  crosses above a given threshold $\nu$. That is, denoting by $\lambda $ the Lebesgue measure on $\mathbb{R}^{d}$,
\begin{equation}M_{0}(\nu; X,D)=\int_{D}1_{\nu}(X(\mathbf{y}))~d\mathbf{y}=\lambda(S_{X,D}(\nu)),\label{mf}
\end{equation}
where
\begin{equation}
S_{X,D}(\nu)=\left\{ \mathbf{z}\in D:\ X(\mathbf{z})\geq \nu\right\}  =  \left\{ \mathbf{z}\in D:\ g(X(\mathbf{z})) = 1 \right\},
\end{equation}
with  $g(x)=1_{\nu}(x)$,  the indicator function based on threshold $\nu$.

In the Gaussian standard case, since $p(u)=\frac{1}{\sqrt{2\pi}}\exp(-u^{2}/2)$,  equation (\ref{ex-0}) leads to
$$1_{\nu}(x)=\sum_{k=0}^{\infty}\frac{G_{k}(\nu)}{k!}\mathcal{H}_{k}(x),$$
where $\{ \mathcal{H}_k, \ k \geq 0 \}$ denotes the basis of (non-normalized) Hermite polynomials, with
$$G_{k}(\nu)=\left\langle 1_{\nu}, \mathcal{H}_{k}\right\rangle_{ L^{2}((a,b), p(u)du)}=  \frac{1}{\sqrt{2\pi}}\exp(-\nu^{2}/2)\mathcal{H}_{k-1}(\nu),\quad k\geq 1,$$
and $G_{0}(\nu)=\frac{1}{\sqrt{2\pi}}\int_{\nu}^{\infty}\exp(-u^{2}/2)du= 1- \Phi(\nu)$ (the value of the decumulative normal distribution at $\nu$). Therefore,
$$M_{0}(\nu; X,D)=\sum_{k=0}^{\infty}\frac{G_{k}(\nu)}{k!}\int_{D}\mathcal{H}_{k}(X(\mathbf{y}))~d\mathbf{y}$$
\noindent (see Leonenko and Ruiz-Medina, 2022).

The following assumption on the large-scale behaviour of the correlation function $\gamma $ is considered:

\medskip

\noindent \textbf{Assumption I.}
\begin{eqnarray}
\gamma (\|\mathbf{x}-\mathbf{y}\|) &=&\mathcal{O}\left(\|\mathbf{x}-\mathbf{y}\|^{-\varrho }\right),\quad \mbox{as} \enspace \|\mathbf{x}-\mathbf{y}\|\to \infty,\quad \varrho \in (0,d).\label{lrd}
\end{eqnarray}

\subsection{Information-complexity measures}
\label{subsubsection:information-complexity}

Since the seminal paper by Shannon (1948), arisen in the context of communications, Information Theory  has extraordinarily grown as a fundamental scientific discipline, with a wide projection in many diverse fields of application. In particular, a variety of information and complexity measures have been proposed and thoroughly studied, with the aim of characterizing the uncertain behaviour inherent to random systems.

Although most concepts have a simpler interpretation in the case of systems with a finite number of states, in this preliminary introduction, and in the sequel, we directly refer to definitions for the continuous case, as is the object of this paper. Accordingly, for a continuous multivariate probability distribution with density function $\{f(\mathbf{x}): \mathbf{x} \in \mathbb{R}^n\}$, Shannon entropy (in this continuous case, also called `differential entropy') is defined as
$$ H(f) := - \int_{\mathbb{R}^n}\ln (f(\mathbf{x})) ~ f(\mathbf{x})~ d\mathbf{x} = E_{f}[-\ln (f)].$$
As is well known, the infimum and supremum of $H(f)$ are $-\infty$ and $+ \infty$, respectively. Shannon entropy satisfies `extensivity', \emph{i. e.} additivity for independent (sub)systems.

Among various generalizations of Shannon entropy,  Rényi entropy (Rényi, 1961), based on a deformation (distortion) parameter,  constitutes the most representative one under preservation of extensivity. For a continuous multivariate distribution with probability density function $\{f(\mathbf{x}): \mathbf{x} \in \mathbb{R}^n\}$, Rényi entropy of order $q$ is defined as
$$ H_q(f) :=  \frac{1}{1-q} \ln \left(\int_{\mathbb{R}^n} f^q(\mathbf{x})~ d\mathbf{x} \right) = \frac{1}{1-q}\ln\left( E_{f}[f^{q-1}]\right) \qquad (q\neq1). $$
As before, the infimum and supremum of $H_q(f)$ are $-\infty$ and $+ \infty$, respectively. Shannon entropy $H(f)$ is the limiting case of Rényi entropy $H_q(f)$ as $q\rightarrow 1$, hence also denoted as $H_1(f)$.

In particular, Rényi entropy constitutes the basis for the formal definition of the two-parameter generalized complexity measures proposed by López-Ruiz et al. (2009), given by
$$
C_{\alpha, \beta}(f) := e^{H_{\alpha}(f)-H_{\beta}(f)},
$$
for $0 < \alpha, \beta < \infty$.

Campbell (1968) justified the interpretation of  Shannon and Rényi entropies in exponential scale (both in the discrete and continuous cases) as an index of `diversity' or `extent' of a distribution. For the continuous case, the diversity index of order $q$,
$$
DI_{q}(f) = e^{H_q(f)},
$$
then varies between 0 and $+\infty$, and
$$
C_{\alpha, \beta}(f) = \frac{DI_{\alpha}(f)}{DI_{\beta}(f)},
$$
which leads to the interpretation of this concept of complexity in terms of sensitivity of the diversity index of order $q$ with respect to the deformation parameter; see Angulo et al. (2021).

Beyond the assessment on global uncertainty, divergence measures are defined for comparison of two given probability distributions at state level. As before, we here directly focus on versions for the continuous case. For two density functions $\left\{f(\mathbf{x}): \mathbf{x} \in \mathbf{R}^n \right\}$ and $\left\{g(\mathbf{x}): \mathbf{x} \in \mathbf{R}^n \right\}$, with $f$ being absolutely continuous with respect to $g$, Kullback and Leibler (1951), following the same conceptual approach leading to the definition of Shannon entropy, introduced the (directed)  divergence of  $f$ from $g$ as
\begin{align}
KL(f\|g) & := \int_{\mathbb{R}^d} f(\mathbf{x})~ \ln \left(\frac{f(\mathbf{x})}{g(\mathbf{x})} \right)~ d\mathbf{x} = E_{f}\left[\ln\left( \frac{f}{g}\right) \right],\nonumber
\end{align}
which, among other uses, has been widely adopted as a meaningful reference measure for inferential optimization purposes. Correspondingly, a generalization is given by Rényi (1961) divergence of order $q$, defined as, for  $q\neq1,$
\begin{align}
H_q(f\|g) & :=  \frac{1}{q-1} \ln \left(\int_{\mathbb{R}^d} f(\mathbf{x})~ \left( \frac{f(\mathbf{x})}{g(\mathbf{x})} \right)^{q-1} d\mathbf{x}\right)  = \frac{1}{q-1} \ln \left( E_{f}\left[\left( \frac{f}{g} \right)^{q-1}\right]\right).\nonumber
\end{align}
For $q \rightarrow 1$, $H_q(f\|g)$ tends to $KL(f\|g)$ (also denoted as $H_1(f\|g)$).

Angulo et al. (2021) proposed a natural formulation of a `relative diversity' index of order $q$ as
$$
DI_q(f\|g) = e^{H_q(f\|g)},
$$
meaning the structural departure of $f$ from $g$ in terms of the state-by-state probability contribution to diversity. This also gives a complementary interpretation, in terms of sensitivity with respect to the deformation parameter of Rényi divergence, for the two-parameter generalized relative complexity measure introduced by Romera et al. (2011):
$$
C_{\alpha, \beta}(f\|g) := e^{H_{\alpha}(f\|g)-H_{\beta}(f\|g)} = \frac{DI_{\alpha}(f\|g)}{DI_{\beta}(f\|g)},
$$
for $0 < \alpha, \beta < \infty$.

A further step, aiming at quantifying stochastic dependence between two random vectors, is achieved in terms of mutual information measures. From the point of view of departure from independence, divergence measures constitute, in particular, a direct instrumental approach, comparing the (true) joint distribution to the product of the corresponding marginal distributions (hypothetical case of independence). Thus, in the continuous case, for two random vectors $X\sim f_X$ and $Y\sim f_Y$, with $(X,Y)\sim f_{XY}$, the Rényi-divergence-based measure of mutual information of order $q$ is defined as
\begin{align}
 I_q(X,Y) & := H_q(f_{XY}\|f_X f_Y),
\end{align}
including the special case
\begin{align}
I(X,Y) & := KL(\bar{p}_{XY}\|\bar{p}_X \bar{p}_Y)=H(X) + H(Y) - H(X,Y)
\end{align}
(with the last equality not being similarly satisfied, in general, for $q \neq 1$). Related concepts and interpretations can be derived in relation to `mutual complexity' and `mutual diversity' (see, \emph{e. g.}, Alonso et al., 2016; Angulo et al., 2021).

These elements are applied in the next sections to studying, under the informational approach, the large-scale asymptotic behaviour of real and infinite-dimensional valued (for the spatio-temporal case) random fields of Lancaster-Sarmanov type.

\section{Methodology: mutual information dependence assessment}
\label{section:methodology-MI}

In this section,  we apply equation (\ref{lrd}) in the derivation of the  asymptotic order characterizing the spatial macroscale  behaviour of mutual information between the marginal spatial  components of  Lancaster-Sarmanov subordinated random fields.
Under \textbf{Assumption I},  this asymptotic order is related to the LRD parameter $\varrho$ of the underlying Lancaster-Sarmanov random field. Note that the lower values of $\varrho$ correspond to  higher asymptotic structural diversity loss. Such an asymptotic order is evaluated in Section \ref{subsection:AASMI}, in particular, for mutual information based on Shannon entropy, which corresponds to a limiting case of the asymptotic spatial macroscale order of mutual information based on R\'enyi entropy.

\subsection{Asymptotic analysis from Shannon mutual information}
\label{subsection:AASMI}

Let  $\{X(\mathbf{x}),\ \mathbf{x}\in \mathbb{R}^d\}$ be an element of  the Lancaster-Sarmanov random field class. From equation  (\ref{210bb}),  mutual information between  component r.v.'s  $X(\mathbf{x})$ and $X(\mathbf{y})$  can be computed as follows:
\begin{eqnarray}&&
I(X(\mathbf{x}),X(\mathbf{y}))=\int_{a}^{b} \int_{a}^{b} p(u,v,\|\mathbf{x}-\mathbf{y}\|)~
 \ln\left(\frac{p(u,v,\|\mathbf{x}-\mathbf{y}\|)}{p(u)~p(v)}\right)~dudv\nonumber\\
 &&=\int_{a}^{b} \int_{a}^{b} p(u)~p(v)~\left[ 1+\sum_{k=1}^{\infty }\gamma^{k}
(\|\mathbf{x}-\mathbf{y}\|)~e_{k}(u)~e_{k}(v)\right]~ \nonumber\\ &&\hspace*{1cm}\times  \ln\left(1+\sum_{k=1}^{\infty }\gamma^{k}
(\|\mathbf{x}-\mathbf{y}\|)~e_{k}(u)~e_{k}(v)\right)~dudv.
 \label{mie}
\end{eqnarray}
The following lemma shows that, under \textbf{Assumption I},  the asymptotic behavior  of $\mathcal{S}_{\varrho }(\|\mathbf{x}-\mathbf{y}\|) := I(X(\mathbf{x}),X(\mathbf{y}))$, when
  $\|\mathbf{x}-\mathbf{y}\|\to \infty $,  involves the  LRD parameter $\varrho$, thus providing an indicator of diversity loss  at macroscale,  with higher values attained  as $\varrho $ gets closer to 0.
\begin{lemma}
\label{lem}
Under \textbf{Assumption I}, the following asymptotic behavior holds:
\begin{equation}
\mathcal{S}_{\varrho }(\|\mathbf{x}-\mathbf{y}\|) = \mathcal{O}\left(\|\mathbf{x}-\mathbf{y}\|^{-\varrho }\right), \quad \mbox{as}\enspace \|\mathbf{x}-\mathbf{y}\|\to \infty.
\label{eqabmi}
\end{equation}

\end{lemma}

\begin{proof}
Note that $p(u)\in L^{2}((a,b), p(u)du)$  admits the series expansion
$$p(u)=\sum_{k=0}^{\infty}C_{k}^{p}e_{k}(u),\quad \mbox{with} \quad C_{k}^{p}=\int_{a}^{b}[p(u)]^{2}~e_{k}(u)~du,\quad k\geq 0.
$$
From equation  (\ref{mie}), applying Taylor series expansion of logarithmic function at a neighborhood of 1, keeping in mind that
$\sum_{k=0}^{\infty}\left[C_{k}^{p}\right]^{2}<\infty,$
 we obtain
\begin{eqnarray}&&\mathcal{S}_{\varrho }(\|\mathbf{x}-\mathbf{y}\|)=
I(X(\mathbf{x}),X(\mathbf{y})) \simeq  \int_{a}^{b}\int_{a}^{b}\left[\sum_{k=0}^{\infty}C_{k}^{p}~e_{k}(u)\right]\left[\sum_{l=0}^{\infty}C_{l}^{p}~e_{l}(v)\right]\nonumber\\
&& \times \left[ 1+\sum_{i=1}^{\infty }\gamma^{i}
(\|\mathbf{x}-\mathbf{y}\|)~e_{i}(u)~e_{i}(v)\right]
\left[\sum_{j=1}^{\infty }\gamma^{j}
(\|\mathbf{x}-\mathbf{y}\|)~e_{j}(u)~e_{j}(v)\right]~dudv\nonumber\\
&&=\sum_{j=1}^{\infty}\gamma^{j}
(\|\mathbf{x}-\mathbf{y}\|)\left[C_{j}^{p}\right]^{2}+\sum_{i=1}^{\infty}\sum_{j=1}^{\infty}\gamma^{i+j}
(\|\mathbf{x}-\mathbf{y}\|)\nonumber\\ &&\hspace*{1cm}\times \int_{a}^{b}\int_{a}^{b}~e_{i}(u)~e_{i}(v)~e_{j}(u)~e_{j}(v)~p(u)~p(v)~dudv\nonumber\\
&&=\sum_{j=1}^{\infty}\gamma^{j}
(\|\mathbf{x}-\mathbf{y}\|)\left[C_{j}^{p}\right]^{2}+\sum_{i=1}^{\infty}\sum_{j=1}^{\infty}\gamma^{i+j}
(\|\mathbf{x}-\mathbf{y}\|)~\delta_{i,j} \nonumber\\
&& = \ \sum_{j=1}^{\infty}\gamma^{j}
(\|\mathbf{x}-\mathbf{y}\|)\left[C_{j}^{p}\right]^{2}+\sum_{i=1}^{\infty}\gamma^{2i}(\|\mathbf{x}-\mathbf{y}\|)\nonumber\\ &\leq & \left( \sup_{j\geq 1}\left[C_{j}^{p}\right]^{2}\right) \sum_{j=1}^{\infty}\gamma^{j}
(\|\mathbf{x}-\mathbf{y}\|)+\sum_{i=1}^{\infty}\gamma^{2i}(\|\mathbf{x}-\mathbf{y}\|) \nonumber\\
&& = \ \left(\sup_{j\geq 1}\left[C_{j}^{p}\right]^{2}\right) \frac{\gamma (\|\mathbf{x}-\mathbf{y}\|)}{1-\gamma (\|\mathbf{x}-\mathbf{y}\|)}+
\frac{\gamma^{2} (\|\mathbf{x}-\mathbf{y}\|)}{1-\gamma^{2} (\|\mathbf{x}-\mathbf{y}\|)}
\nonumber\\
&& =
\mathcal{O}\left(\|\mathbf{x}-\mathbf{y}\|^{-\varrho }\right),\quad \mbox{as}\enspace \|\mathbf{x}-\mathbf{y}\|\to \infty.
\nonumber\\
\label{ord}
\end{eqnarray}

Similarly,  as  $\|\mathbf{x}-\mathbf{y}\|\to \infty,$
\begin{eqnarray}\mathcal{S}_{\varrho }(\|\mathbf{x}-\mathbf{y}\|)&=&
I(X(\mathbf{x}),X(\mathbf{y})) \simeq\sum_{j=1}^{\infty}\gamma^{j}
(\|\mathbf{x}-\mathbf{y}\|)\left[C_{j}^{p}\right]^{2}+\sum_{i=1}^{\infty}\gamma^{2i}(\|\mathbf{x}-\mathbf{y}\|)\nonumber\\
& \geq & \sum_{j=1}^{\infty}\gamma^{j}
(\|\mathbf{x}-\mathbf{y}\|)\left[C_{j}^{p}\right]^{2}\nonumber\\
& \geq & \left(\inf_{j\geq 1}\left[C_{j}^{p}\right]^{2}\right)\sum_{j=1}^{\infty}\gamma^{j}
(\|\mathbf{x}-\mathbf{y}\|) =
\mathcal{O}\left(\|\mathbf{x}-\mathbf{y}\|^{-\varrho }\right),
\nonumber\\
\label{ordb}
\end{eqnarray}
\noindent where we have applied that the  rank  of  function $p\in L^{2}((a,b), p(u)du)$ is equal to 1.
\end{proof}

    For $g\in L^{2}((a,b), p(u)du)$, a similar asymptotic behavior is displayed by mutual information
   $I(g(X(\mathbf{x})),g(X(\mathbf{y})))$ when
  $\|\mathbf{x}-\mathbf{y}\|\to \infty $,
  involving the LRD parameter $\varrho$ scaled by the rank $m$ of function $g$ in the orthonormal basis $\{e_{k},\ k\geq 0\}$ (Hermite and Laguerre ranks in the Gaussian and Gamma-correlated cases, respectively).   This fact is proved in the following result.

 \begin{theorem} \label{th1}  Let   $p(u)$ be, as before, the  probability density  characterizing  the marginal probability distributions of the   \emph{Lancaster-Sarmanov} random field   $X=\{ X(\mathbf{z}),\ \mathbf{z}\in \mathbb{R}^{d}\}$.
  Assume that   $g\in L^{2}((a,b), p(u)du)$ has rank $m$, and is such that   $g(X(\mathbf{z}))$    is a discrete random variable whose state space is finite (with cardinal $N$), for $\mathbf{z}\in \mathbb{R}^{d}.$   The following asymptotic behavior  then holds:
   \begin{equation} I\left(g(X(\mathbf{x})),g(X(\mathbf{y}))\right)
  = \mathcal{O}\left(\|\mathbf{x}-\mathbf{y}\|^{-m\varrho }\right),\quad \mbox{as} \enspace  \|\mathbf{x}-\mathbf{y}\|\to \infty.
\label{eqabmi2}
\end{equation}
  \end{theorem}
  \begin{remark}
Note that applying  variable change  theorem, equation (\ref{eqabmi2}) is  also satisfied  in the case where, for every  $\mathbf{z}\in \mathbb{R}^{d},$  $g(X(\mathbf{z}))$ is a continuous random variable  and $g$ admits an inverse function $g^{-1}$ having  non-null derivatives over $g((a,b)).$ Identity (\ref{eqabmi2}) also holds when $g$ is non injective but  a countable set  $\{g_{k}^{-1}(\mathbf{y}),\ k\in \mathbb{N}\}$ of preimages is associated  with every point $\mathbf{y}\in g((a,b)),$   with $g_{k}^{-1}$ having  non-null derivatives over $g((a,b)),$  for $k\in \mathbb{N}.$
\end{remark}

  \begin{proof} Under \textbf{Assumption I}, applying Jensen's  inequality, since $\{e_{k},\ k\geq 0\}$ is an  orthonormal basis of the space
$L^{2}((a,b), p(u)du),$
  \begin{eqnarray}&&I\left(g(X(\mathbf{x})),g(X(\mathbf{y}))\right)
   =  \sum_{i=1}^{N}\sum_{j= 1 }^{N}\ln\left[\frac{p_{g(X(\mathbf{x})),g(X(\mathbf{y}))}(g_{i},g_{j})}{p_{g(X(\mathbf{x}))}(g_{i})~p_{g(X(\mathbf{y}))}(g_{j})}
    \right]\nonumber\\
    && \times \int_{g^{-1}(g_{i})\times g^{-1}(g_{j})}p(u)~p(v) ~\left[ 1+\sum_{h=m}^{\infty }\gamma^{h}
(\|\mathbf{x}-\mathbf{y}\|)~e_{h}(u)~e_{h}(v)\right]~dudv\nonumber\\
    && \leq  \sup_{i,j} \left(\ln\left[\frac{p_{g(X(\mathbf{x})),g(X(\mathbf{y}))}(g_{i},g_{j})}{p_{g(X(\mathbf{x}))}(g_{i})~p_{g(X(\mathbf{y}))}(g_{j})}
    \right]\right)\nonumber\\
&&\times \enspace N^{2}\left[ 1 +\sum_{h=m}^{\infty }\gamma^{h}
(\|\mathbf{x}-\mathbf{y}\|)\int_{a}^{b}\int_{a}^{b} p(u)~p(v)~
e_{h}(u)~e_{h}(v)~dudv\right]
\nonumber\\
&& =  \sup_{i,j}\left(\ln\left[\frac{p_{g(X(\mathbf{x})),g(X(\mathbf{y}))}(g_{i},g_{j})}{p_{g(X(\mathbf{x}))}(g_{i})~p_{g(X(\mathbf{y}))}(g_{j})}
    \right]\right)\nonumber\end{eqnarray}
    \begin{eqnarray}
&&\hspace*{1cm}\times  N^{2}\left[ 1 +\sum_{h=m}^{\infty }\gamma^{h}
(\|\mathbf{x}-\mathbf{y}\|)\left(E_{p}\left[e_{h}(\cdot)\right]\right)^{2}\right]
\nonumber\\
& & \leq \sup_{i,j}\left(\ln\left[\frac{p_{g(X(\mathbf{x})),g(X(\mathbf{y}))}(g_{i},g_{j})}{p_{g(X(\mathbf{x}))}(g_{i})~p_{g(X(\mathbf{y}))}(g_{j})}
    \right]\right)\nonumber\\
&&\hspace*{1cm}\times N^{2}\left[ 1 +\sum_{h=m}^{\infty }\gamma^{h}
(\|\mathbf{x}-\mathbf{y}\|)\|e_{h}\|_{L^{2}((a,b), p(u)du)}^{2}\right]\nonumber\\
& &=  \sup_{i,j}\left(\ln\left[\frac{p_{g(X(\mathbf{x})),g(X(\mathbf{y}))}(g_{i},g_{j})}{p_{g(X(\mathbf{x}))}(g_{i})~p_{g(X(\mathbf{y}))}(g_{j})}
    \right]\right)\nonumber\\
&&\hspace*{1cm}\times N^{2}\left[ \gamma^{0}
(\|\mathbf{x}-\mathbf{y}\|) +\sum_{h=m}^{\infty }\gamma^{h}
(\|\mathbf{x}-\mathbf{y}\|)\right]\nonumber\\
&& =
\mathcal{O}\left(\|\mathbf{x}-\mathbf{y}\|^{-m\varrho }\right),\quad \mbox{as} \enspace \|\mathbf{x}-\mathbf{y}\|\to \infty.
\label{abth1}
\end{eqnarray}
\noindent
From (\ref{210bb}), the following  lower bound  is obtained:
 \begin{eqnarray}&&I\left(g(X(\mathbf{x})),g(X(\mathbf{y}))\right)
   =  \sum_{i=1}^{N}\sum_{j= 1 }^{N}\ln\left[\frac{p_{g(X(\mathbf{x})),g(X(\mathbf{y}))}(g_{i},g_{j})}{p_{g(X(\mathbf{x}))}(g_{i})~p_{g(X(\mathbf{y}))}(g_{j})}
    \right]\nonumber\\
    &&\times \int_{g^{-1}(g_{i})\times g^{-1}(g_{j})}p(u)~p(v)~\left[ 1+\sum_{h=m}^{\infty }\gamma^{h}
(\|\mathbf{x}-\mathbf{y}\|)~e_{h}(u)~e_{h}(v)\right]~dudv\nonumber\\
   & & \geq  \inf_{i,j}\left(\ln\left[\frac{p_{g(X(\mathbf{x})),g(X(\mathbf{y}))}(g_{i},g_{j})}{p_{g(X(\mathbf{x}))}(g_{i})~p_{g(X(\mathbf{y}))}(g_{j})}
    \right]\right)\nonumber\\
&&\times \enspace N^{2}\left[ 1 + \left[\sum_{h=m}^{\infty }\gamma^{h}
(\|\mathbf{x}-\mathbf{y}\|)\right]
\int_{\mathcal{S}\times \mathcal{S}}e_{h}(u)~e_{h}(v)~p(u)~p(v)~dudv\right]
\nonumber\\
&& =  \inf_{i,j}\left(\ln\left[\frac{p_{g(X(\mathbf{x})),g(X(\mathbf{y}))}(g_{i},g_{j})}{p_{g(X(\mathbf{x}))}(g_{i})~p_{g(X(\mathbf{y}))}(g_{j})}
    \right]\right)\nonumber\\
&&\times \enspace N^{2}\left[ \left(E_{p}[1_{\mathcal{S}}e_{h}(\cdot)]\right)^{2} + \left[\sum_{h=m}^{\infty }\gamma^{h}
(\|\mathbf{x}-\mathbf{y}\|)\right]\left(E_{p}[1_{\mathcal{S}}e_{h}(\cdot)]\right)^{2}\right]
\nonumber\\
& &\geq  \inf_{i,j}\left(\ln\left[\frac{p_{g(X(\mathbf{x})),g(X(\mathbf{y}))}(g_{i},g_{j})}{p_{g(X(\mathbf{x}))}(g_{i})~p_{g(X(\mathbf{y}))}(g_{j})}
    \right]\right)\nonumber\\
&&\times \enspace \left[N\left(\inf_{h\geq m}E_{p}[1_{\mathcal{S}}e_{h}(\cdot)]\right)\right]^{2}\left[ \gamma^{0}
(\|\mathbf{x}-\mathbf{y}\|)+ \sum_{h=m}^{\infty }\gamma^{h}
(\|\mathbf{x}-\mathbf{y}\|)\right]
\nonumber\\
& & = \mathcal{O}\left(\|\mathbf{x}-\mathbf{y}\|^{-m\varrho }\right),\quad \mbox{as} \enspace \|\mathbf{x}-\mathbf{y}\|\to \infty,
\label{abth1b}
\end{eqnarray}
\noindent where  $\mathcal{S}\subseteq (a,b)$ satisfies $p(\mathcal{S})=\inf_{i=1,\dots ,N} p(g^{-1}(g_{i}))>0 ,$  and in the last inequality we have applied that $1_{\mathcal{S}}\in  L^{2}((a,b), p(u)du)$  with $\|1_{\mathcal{S}}\|_{L^{2}((a,b), p(u)du)}>0,$  and
$E_{p}[1_{\mathcal{S}}e_{m}(\cdot)]>0$, $\forall m$, hence $\inf_{h\geq m}E_{p}[1_{\mathcal{S}}e_{h}(\cdot)]>0.$
 From  (\ref{abth1}) and (\ref{abth1b}), equation (\ref{eqabmi2}) holds.
  \end{proof}

\subsection{Asymptotic analysis from R\'enyi mutual information}
\label{subsection:AARMI}

  This section extends the asymptotic analysis derived in the previous section beyond the special limit case based on Shannon entropy to the framework of R\'enyi mutual information. Hence, several interpretations arise in terms Campbell's diversity indices for generalized complexity  measures.

From equation  (\ref{210bb}),  for any $\mathbf{x},\mathbf{y}\in \mathbb{R}^d$, R\'enyi mutual information between  random components $X(\mathbf{x})$ and $X(\mathbf{y})$ of random field $X$ in the Lancaster-Sarmanov class is given by
\begin{eqnarray}
&& I_{q}(X(\mathbf{x}),X(\mathbf{y}))=\frac{1}{q-1}\ln\left(\int_{a}^{b}\int_{a}^{b}\left[\frac{p(u,v,\|\mathbf{x}-\mathbf{y}\|)}{p(u)~p(v)}\right]^{q-1}
\right.\nonumber\\ &&\hspace*{5cm}\times \left.p(u,v,\|\mathbf{x}-\mathbf{y}\|)~dudv\right)
\label{rmi}
\end{eqnarray}
\begin{theorem}
\label{th2} Under  the conditions assumed in Lemma \ref{lem}, if
 \begin{equation}\sup_{k_{1},\dots, k_{q}\in \mathbb{N}^{q}}\left(E_{p}\left[\prod_{i=1}^{q}e_{k_{i}}\right]\right)^{2}<\infty ,\label{cth2}
\end{equation}
\noindent then,
\begin{equation*}
I_{q}(X(\mathbf{x}),X(\mathbf{y}))=
\mathcal{O}\left(\|\mathbf{x}-\mathbf{y}\|^{-q\varrho }\right),\quad \mbox{as} \enspace \|\mathbf{x}-\mathbf{y}\|\to \infty.\end{equation*}
Here, as before, $E_{p}[\cdot]$ denotes the expectation with respect to the marginal probability density $p.$
\end{theorem}
\begin{proof}
Under  the conditions of Lemma \ref{lem}, from equations (\ref{210bb}) and  (\ref{rmi}),  for any $\mathbf{x},\mathbf{y}\in \mathbb{R}^d$,
\begin{eqnarray}&&
I_{q}(X(\mathbf{x}),X(\mathbf{y}))  =  \frac{1}{q-1}\ln\left(\int_{a}^{b}\int_{a}^{b}\left[1+\sum_{k=1}^{\infty }\gamma^{k}
(\|\mathbf{x}-\mathbf{y}\|)~e_{k}(u)~e_{k}(v)\right]^{q}\right.\nonumber\\ &&\hspace*{3cm}\left.\times p(u)~p(v)~dudw\right)\nonumber\\ &&= \frac{1}{q-1}\ln \left(1+\int_{a}^{b}\int_{a}^{b}
|\mathcal{Q}_{\|\mathbf{x}-\mathbf{y}\|}(u,v)|^{q}p(u)~p(v)~dudv\right)\nonumber\\
&&=\frac{1}{q-1}\ln\left(1+\sum_{k_{1},\dots,k_{q}}\gamma^{k_{1}+\dots+k_{q}}\left(\|\mathbf{x}-\mathbf{y}\|\right)\right.
\nonumber\\ &&\left.\times \int_{a}^{b}\int_{a}^{b}\prod_{i=1}^{q}e_{k_{i}}\otimes  e_{k_{i}}(u,v) ~ p\otimes p(u,v)~dudv\right)\nonumber\end{eqnarray}
\begin{eqnarray}
&&=\frac{1}{q-1}\ln\left(1+\sum_{k_{1},\dots,k_{q}}\gamma^{k_{1}+\dots+k_{q}}\left(\|\mathbf{x}-\mathbf{y}\|\right)
\left[E_{p}\left[\prod_{i=1}^{q}e_{k_{i}}\right]\right]^{2}\right)\nonumber\\
&&\simeq  \frac{1}{q-1}\left(\sum_{k_{1},\dots,k_{q}}\gamma^{k_{1}+\dots+k_{q}}\left(\|\mathbf{x}-\mathbf{y}\|\right)
\left[E_{p}\left[\prod_{i=1}^{q}e_{k_{i}}\right]\right]^{2}\right)
\nonumber\\
 &&\leq  \frac{1}{q-1}\sup_{k_{1},\dots, k_{q}\in \mathbb{N}^{q}}\left(E_{p}\left[\prod_{i=1}^{q}e_{k_{i}}\right]\right)^{2}\left[\frac{\gamma (\|\mathbf{x}-\mathbf{y}\|)}{1-\gamma (\|\mathbf{x}-\mathbf{y}\|)}\right]^{q}\nonumber\\
&&=
\mathcal{O}\left(\|\mathbf{x}-\mathbf{y}\|^{-q\varrho }\right),\quad \mbox{as} \enspace \|\mathbf{x}-\mathbf{y}\|\to \infty.
\nonumber\\
\end{eqnarray}

Here, $\mathcal{Q}_{\|\mathbf{x}-\mathbf{y}\|}(u,v)$ denotes the kernel $\sum_{k=1}^{\infty }\gamma^{k}
(\|\mathbf{x}-\mathbf{y}\|)~e_{k}(u)~e_{k}(v),$  for  $u,v\in (a,b)$.

Similarly,  the following asymptotic  behavior is obtained:
\begin{eqnarray}
I_{q}(X(\mathbf{x}),X(\mathbf{y})) & \simeq & \frac{1}{q-1}\left(\sum_{k_{1},\dots,k_{q}}\gamma^{k_{1}+\dots+k_{q}}\left(\|\mathbf{x}-\mathbf{y}\|\right)
\left[E_{p}\left[\prod_{i=1}^{q}e_{k_{i}}\right]\right]^{2}\right)\nonumber\\
& \geq & \frac{1}{q-1}\inf_{k_{1},\dots, k_{q}\in \mathbb{N}^{q}}\left[E_{p}\left[\prod_{i=1}^{q}e_{k_{i}}\right]\right]^{2}\left[\frac{\gamma (\|\mathbf{x}-\mathbf{y}\|)}{1-\gamma (\|\mathbf{x}-\mathbf{y}\|)}\right]^{q}\nonumber\\
&=&
\mathcal{O}\left(\|\mathbf{x}-\mathbf{y}\|^{-q\varrho }\right),\quad \|\mathbf{x}-\mathbf{y}\|\to \infty,
\nonumber\\
\end{eqnarray}
as we wanted to prove. Here, we have   applied that $\inf_{k_{1},\dots, k_{q}\in \mathbb{N}^{q}}\left[E_{p}\left[\prod_{i=1}^{q}e_{k_{i}}\right]\right]^{2}>0,$  for $q>1.$  Note that for $q=1,$ $E_{p}[e_{k}]=0,$ for every $k\geq 1$ (see, e.g.,  Leonenko at al., 2017).
\end{proof}

\begin{remark}
Condition (\ref{cth2}) is satisfied, for example,  when the moment generating function of the marginal probability distributions exists. That is the case of Gaussian and Gamma-correlated random fields.
\end{remark}

\subsection{Simulations}
\label{subsection:spatial-simulations}

Let $\{X(\mathbf{x}),\ \mathbf{%
x}\in \mathbb{R}^{d}\}$ be  a measurable zero-mean Gaussian homogeneous and
isotropic mean-square continuous random field on a probability space $%
(\Omega ,\mathcal{A},P),$ with $\mathrm{E}[Y^{2}(\mathbf{x})]=1,$ for all $%
\mathbf{x}\in \mathbb{R}^{d},$ and correlation function $\mathrm{E}[X(%
\mathbf{x})X(\mathbf{y})]$  $=B(\Vert \mathbf{x}-\mathbf{y}\Vert )$ of the form:
\begin{equation}
B(\Vert \mathbf{z}\Vert )=\frac{\mathcal{L}(\Vert \mathbf{z}\Vert
)}{\Vert \mathbf{z}\Vert ^{\alpha }},\quad \mathbf{z}\in
\mathbb{R}^{d},\quad 0<\alpha < d/2.  \label{cov}
\end{equation}%
\noindent The correlation  $B$ of $X$ is
a continuous function of $r=\Vert \mathbf{z}\Vert.$ It then follows that $\mathcal{L}(r)=\mathcal{O}(r^{\alpha }),$ $%
r\to 0.$ Note that the covariance function

\begin{equation}
B(\Vert \mathbf{z}\Vert )=\frac{1}{(1+\Vert \mathbf{z}\Vert ^{\beta
})^{\gamma }},\quad 0<\beta \leq 2,\quad \gamma >0,\label{eqcorrfunctex}
\end{equation}%
\noindent is a particular case of the family of covariance functions
 (\ref{cov}) studied here with $\alpha =  \beta\gamma,$ and
\begin{equation}\mathcal{L}(\Vert
\mathbf{z}\Vert )=\frac{\Vert \mathbf{z}\Vert ^{\beta \gamma }}{(1+\Vert \mathbf{z}%
\Vert ^{\beta })^{\gamma }}.\label{svfex}
\end{equation}

In the present  simulation study, we restrict our attention to such a family of covariance functions. Specifically,  we have considered the parameter values $\beta = \gamma =0.2,$ in the generations of Gaussian random field $X$ with covariance function
(\ref{eqcorrfunctex}) (see Figure \ref{fig:62}).  From ten independent copies  $X_{i},$ $i=1,\dots,10,$ of  random field  $X$ a $\chi_{10}^{2}$ random field is also generated from the identity
$$\chi_{10}^{2}(\mathbf{x})=\frac{1}{2}\left[X_{1}^{2}(\mathbf{x})+\dots+X_{10}^{2}(\mathbf{x})\right],\mathbf{x}\in \mathbb{R}^{d}.$$
\noindent Its correlation function $\gamma $ is given  by
$$\gamma (\|\mathbf{x}-\mathbf{y}\|)=\frac{\mbox{Cov}(\chi_{10}^{2}(\mathbf{x}), \chi_{10}^{2}(\mathbf{y}))}{\mbox{Var}(\chi_{10}^{2}(\mathbf{0}))}=B^{2}(\Vert \mathbf{z}\Vert ),$$
\noindent where $B^{2}(\Vert \mathbf{z}\Vert )$ has been introduced in  (\ref{eqcorrfunctex}).

The results derived in Lemma \ref{lem}, and Theorems \ref{th1}  and \ref{th2}  are illustrated for both models, computing Shannon and R\'enyi mutual informations for the corresponding  original random variables, and for their transformed versions in terms of indicator functions, considering an increasing sequence of distances between the involved random variables. Specifically,  a truncated version of $I(X(\mathbf{x}),X(\mathbf{y})),$ and  $I_{q}(X(\mathbf{x}),X(\mathbf{y})),$ based on $M=5$ Hermite and Laguerre polynomials, respectively,  is computed for spatial distances $d_{n}=\|\mathbf{x}_{n}-\mathbf{y}_{n}\|,$  $d_{n}=1,\dots,1000.$  The derived lower and upper bounds are  represented as well.
  In particular, Figure \ref{fig:62b} displays in green dashed line   $I(X(\mathbf{x}),X(\mathbf{y}))$  (top-left), and $I(\chi_{10}^{2}(\mathbf{x}),\chi_{10}^{2}(\mathbf{y}))$ (top-right)\linebreak
$I(g(X(\mathbf{x})),g(X(\mathbf{y})))$  (bottom-left), and $I(g(\chi_{10}^{2}(\mathbf{x})),g(\chi_{10}^{2}(\mathbf{y})))$ \linebreak (bottom-right). Here, $g(x)=1_{\nu}(x),$ $\nu=0.95.$
 The upper and lower bounds are represented in dashed  red and blue lines, respectively.
The  values of $I_{q}(X(\mathbf{x}),X(\mathbf{y})),$ $q=1.5,~2.10,~2.25;$ $I_{q}(\chi_{10}^{2}(\mathbf{x}),\chi_{10}^{2}(\mathbf{y})),$    $q=2,~2.05,~2.10;$ $I_{q}(g(X(\mathbf{x})),g(X(\mathbf{y}))),$ $q=1.5,~1.75,~1.95;$    and $I_{q}(g(\chi_{10}^{2}(\mathbf{x})),g(\chi_{10}^{2}(\mathbf{y}))),$   $q=1.75,~1.85,~ 1.95;$ for $g(x)=1_{\nu}(x),$ $\nu=0.95,$ are also plotted in Figure \ref{fig:62bb}.  It must be observed that,  as we have checked through a large number of simulations, sensitivity at shorter distances with respect to the deformation parameter $q$ depends on  the polynomial basis and the truncation order selected.

\begin{figure}[htbp]
\begin{center}
 \hspace*{1.2cm}
 \includegraphics[height=5cm, width=5.5cm]{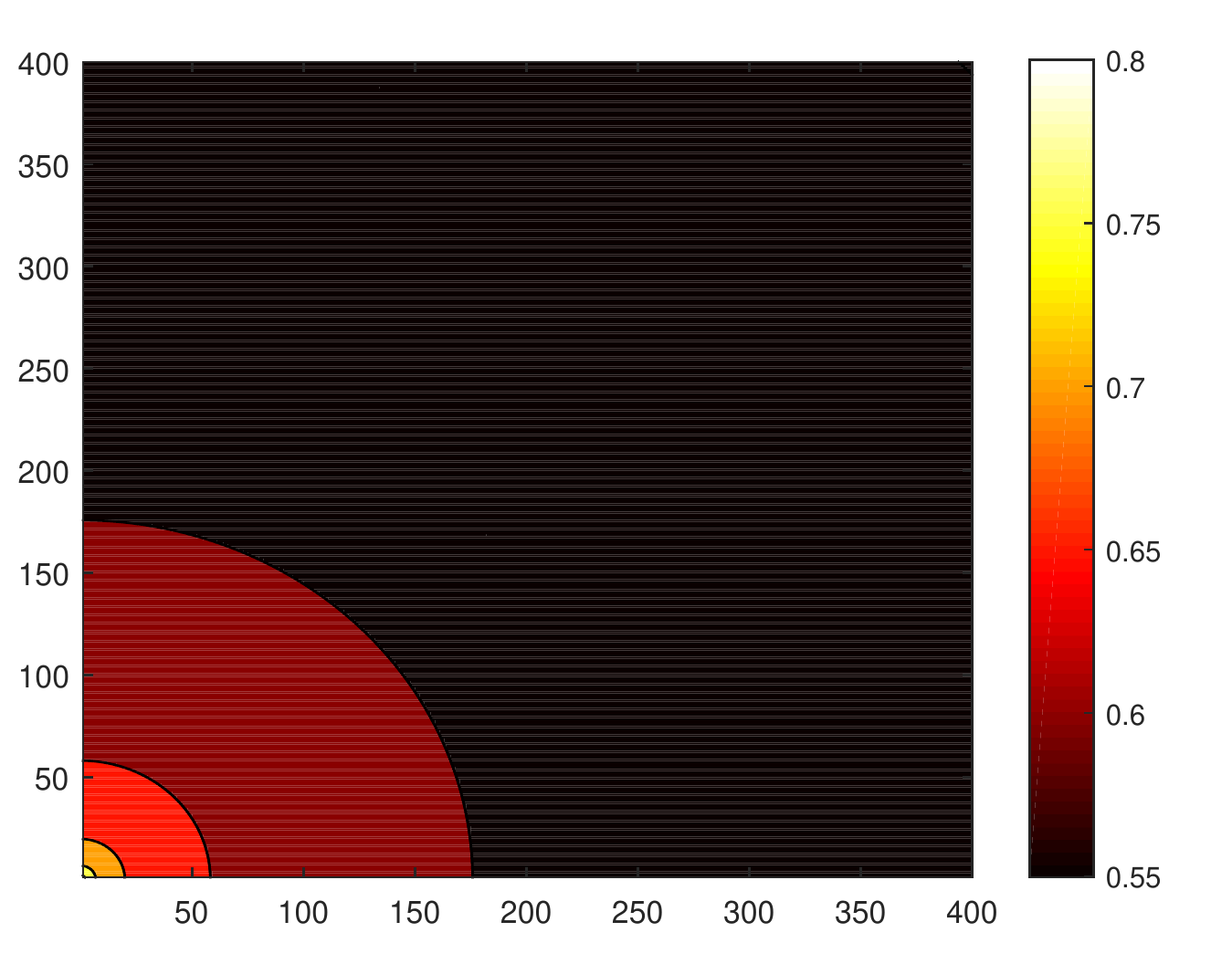}
 \includegraphics[height=5cm, width=5.5cm]{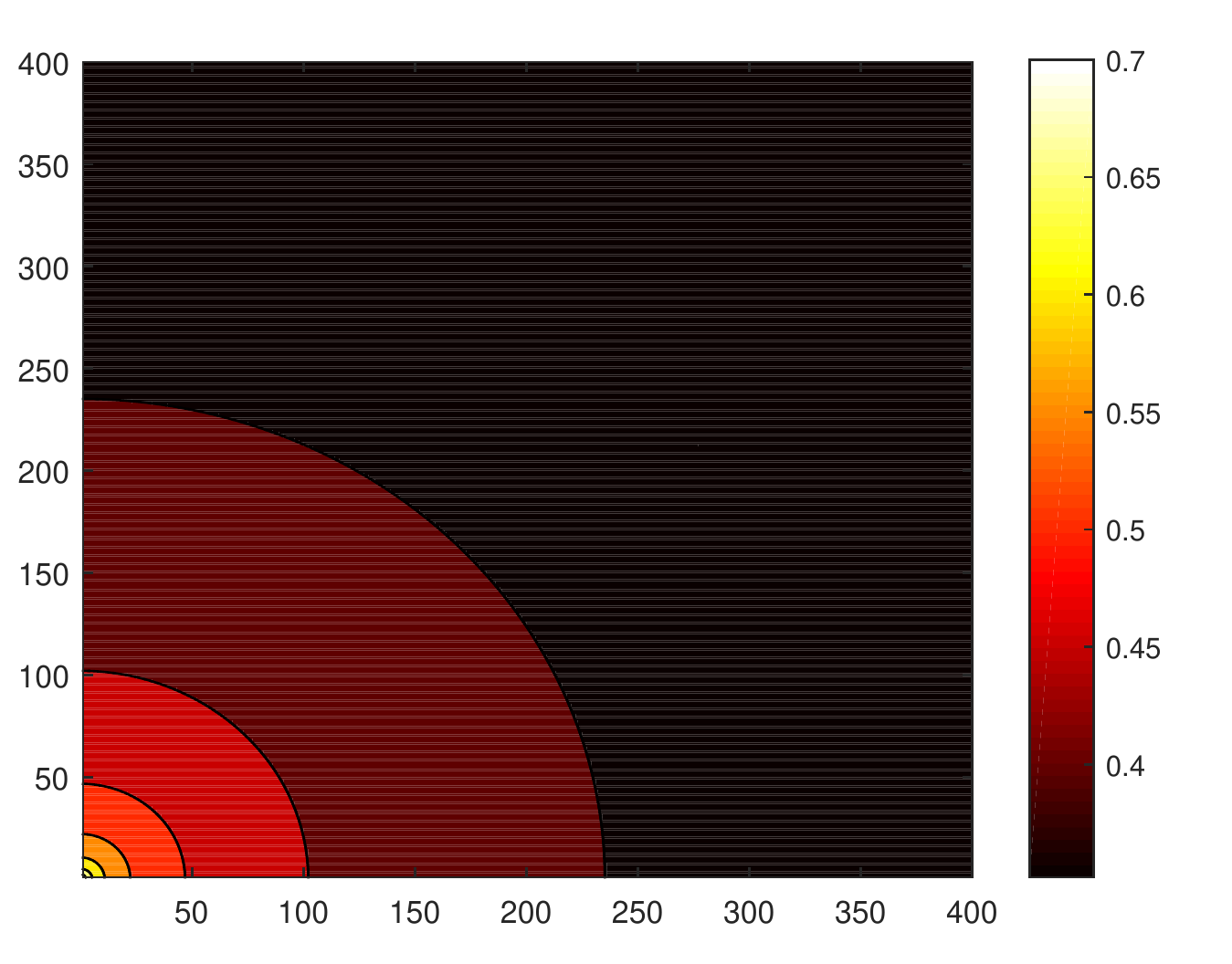}
\end{center}
\caption{{Covariance function of $X$  (left-hand side), and covariance function of $\chi_{10}^{2}$  (right-hand side) for a regular grid of $20\times 20$ spatial locations, with respect to distances in horizontal axis.\label{fig:62}}}
\end{figure}

\begin{figure}[htbp]
\begin{center}
 \hspace*{1.2cm}
 \hspace*{-0.5cm}\includegraphics[height=5cm, width=5cm]{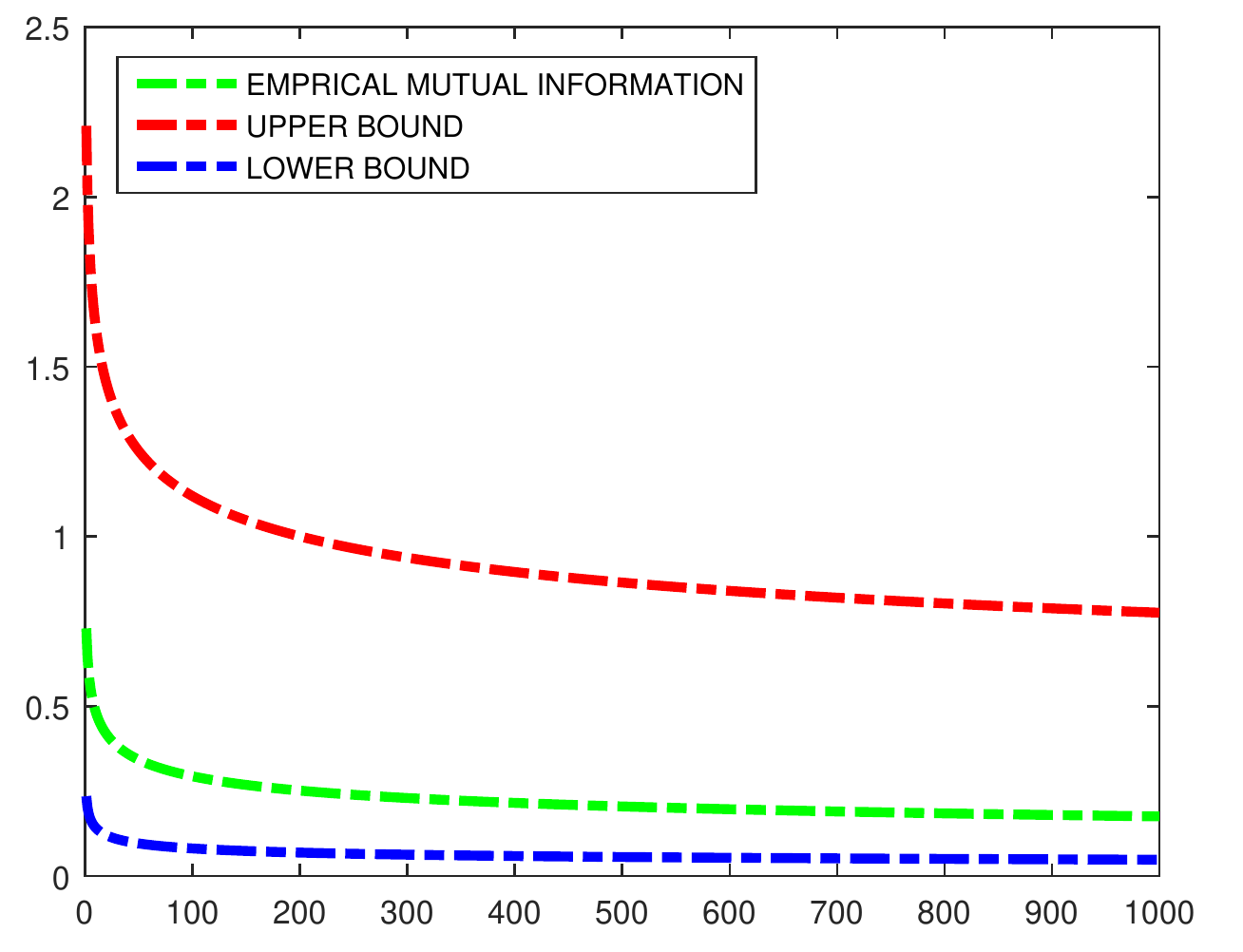}
\includegraphics[height=5cm, width=5cm]{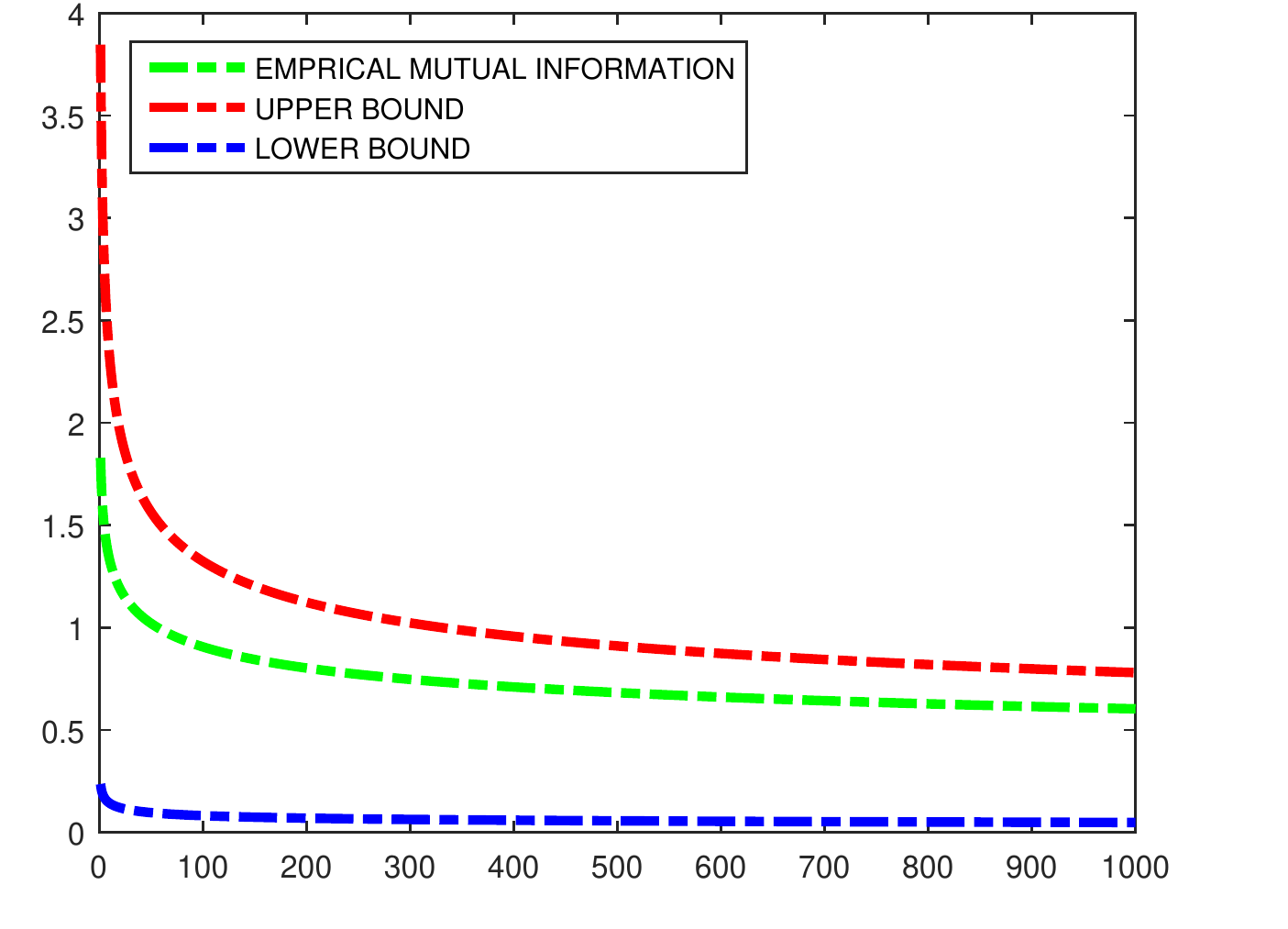}
\hspace*{0.5cm} \includegraphics[height=5cm, width=5cm]{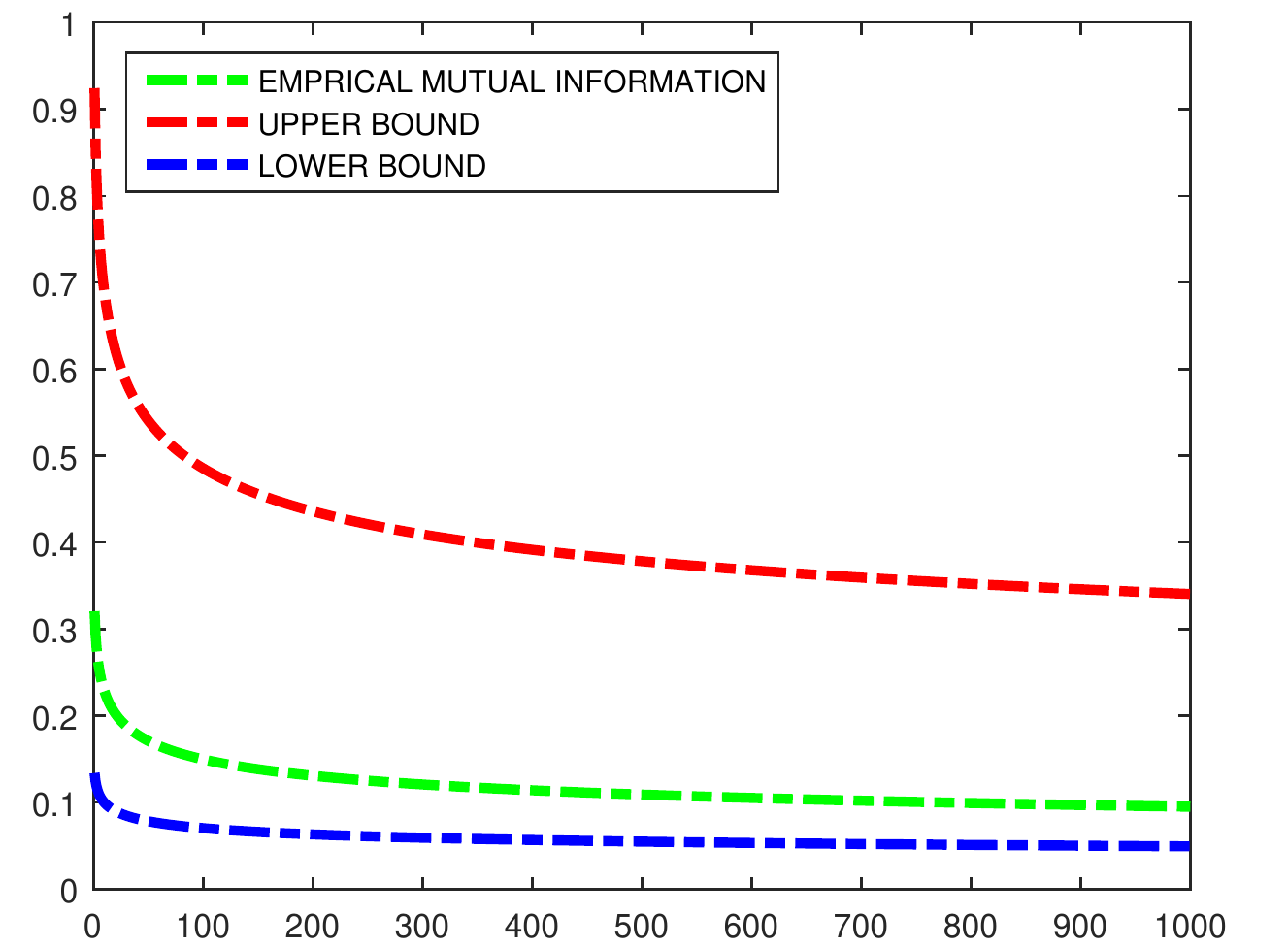}
\includegraphics[height=5cm, width=5cm]{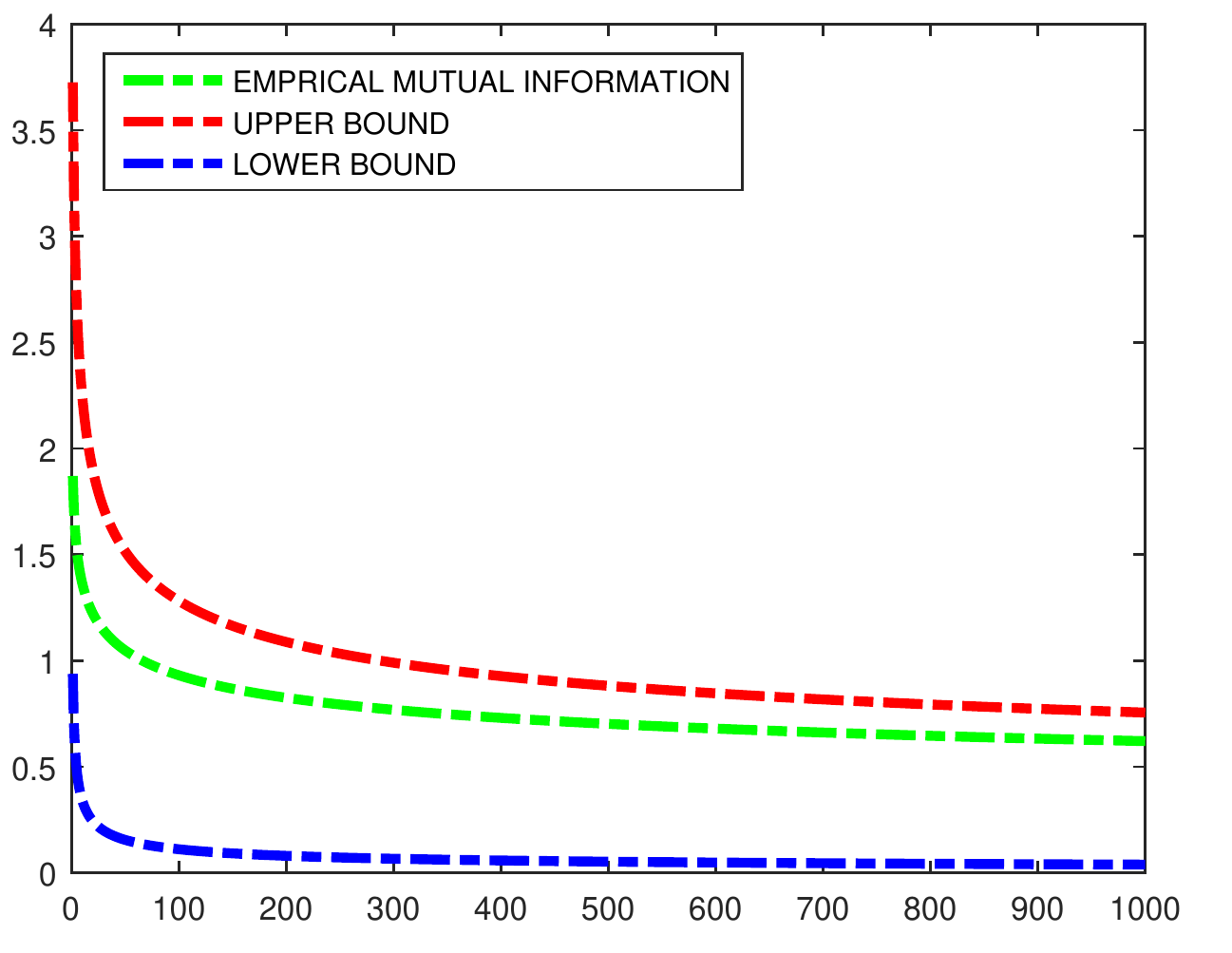}
\end{center}
\caption{{In green color, truncated $I(X(\mathbf{x}),X(\mathbf{y}))$  (top-left), and $I(\chi_{10}^{2}(\mathbf{x}),\chi_{10}^{2}(\mathbf{y}))$ (top-right),
$I(g(X(\mathbf{x})),g(X(\mathbf{y})))$  (bottom-left), and $I(g(\chi_{10}^{2}(\mathbf{x})),g(\chi_{10}^{2}(\mathbf{y})))$ (bottom-right).
 The upper and lower bounds are represented in red and blue colors, respectively.  Here, $g(x)=1_{\nu}(x),$ $\nu=0.95$\label{fig:62b}}}
\end{figure}

\begin{figure}[htbp]
\begin{center}
\includegraphics[height=4cm, width=4cm]{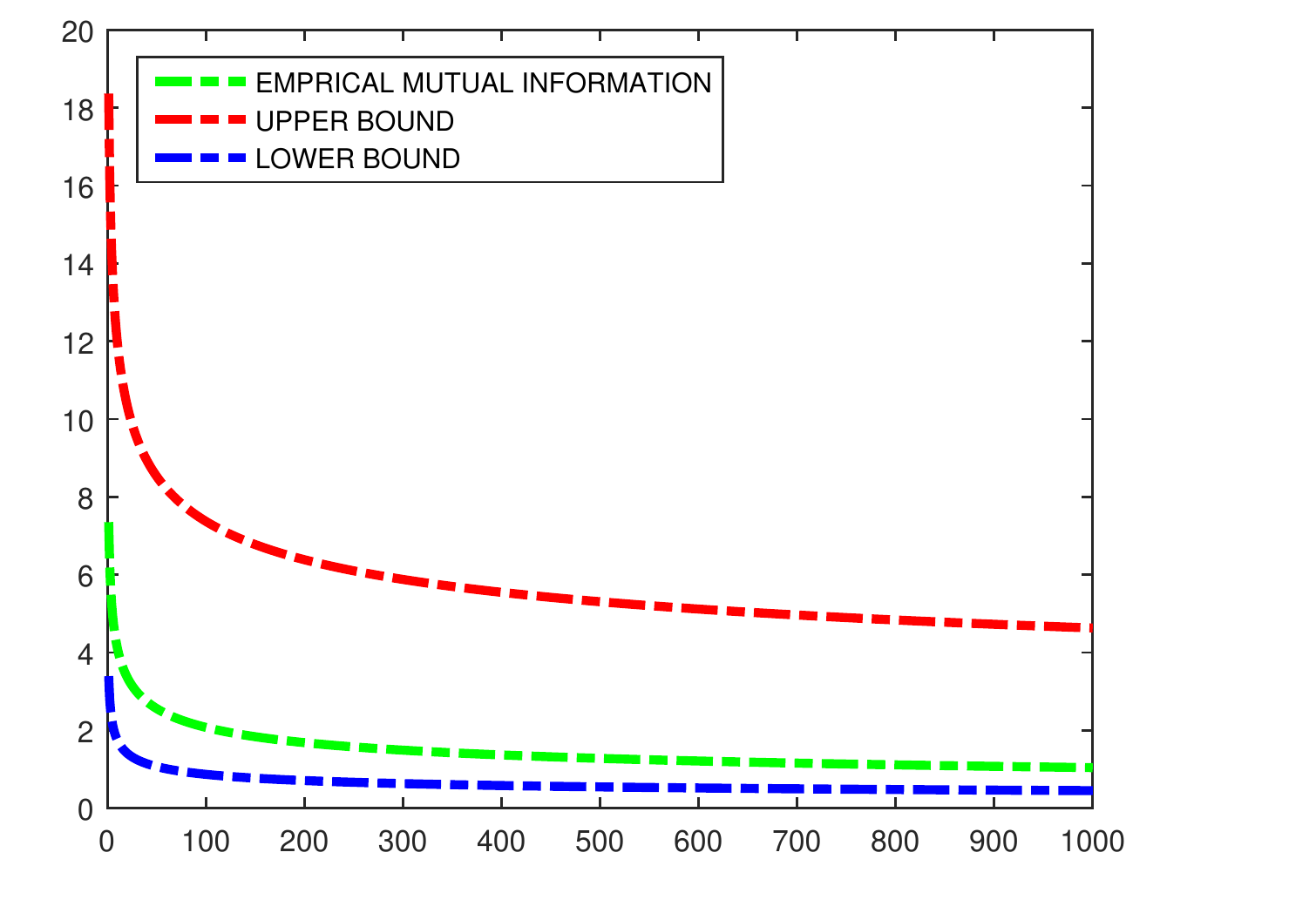}
\includegraphics[height=4cm, width=4cm]{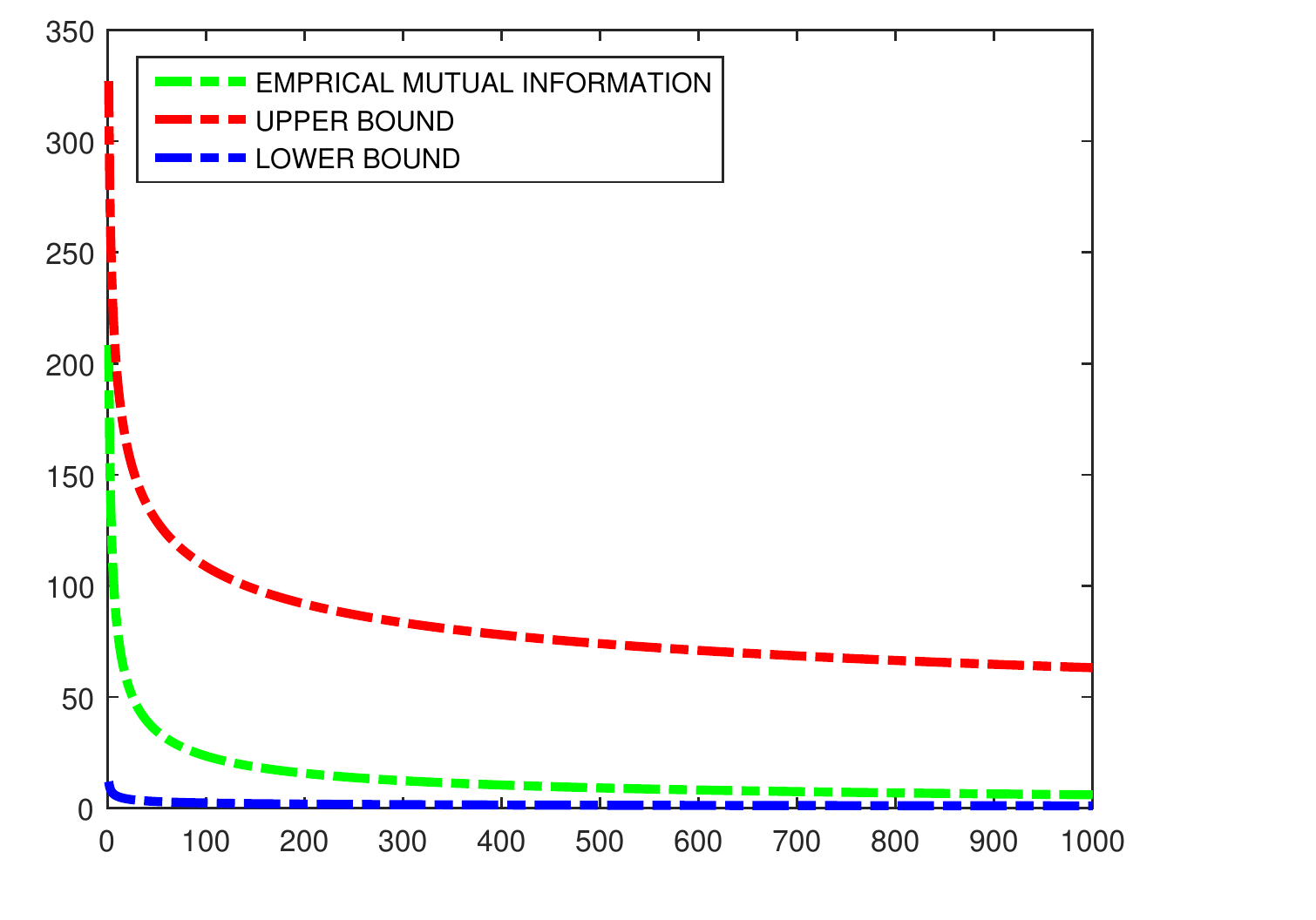}
 \includegraphics[height=4cm, width=4cm]{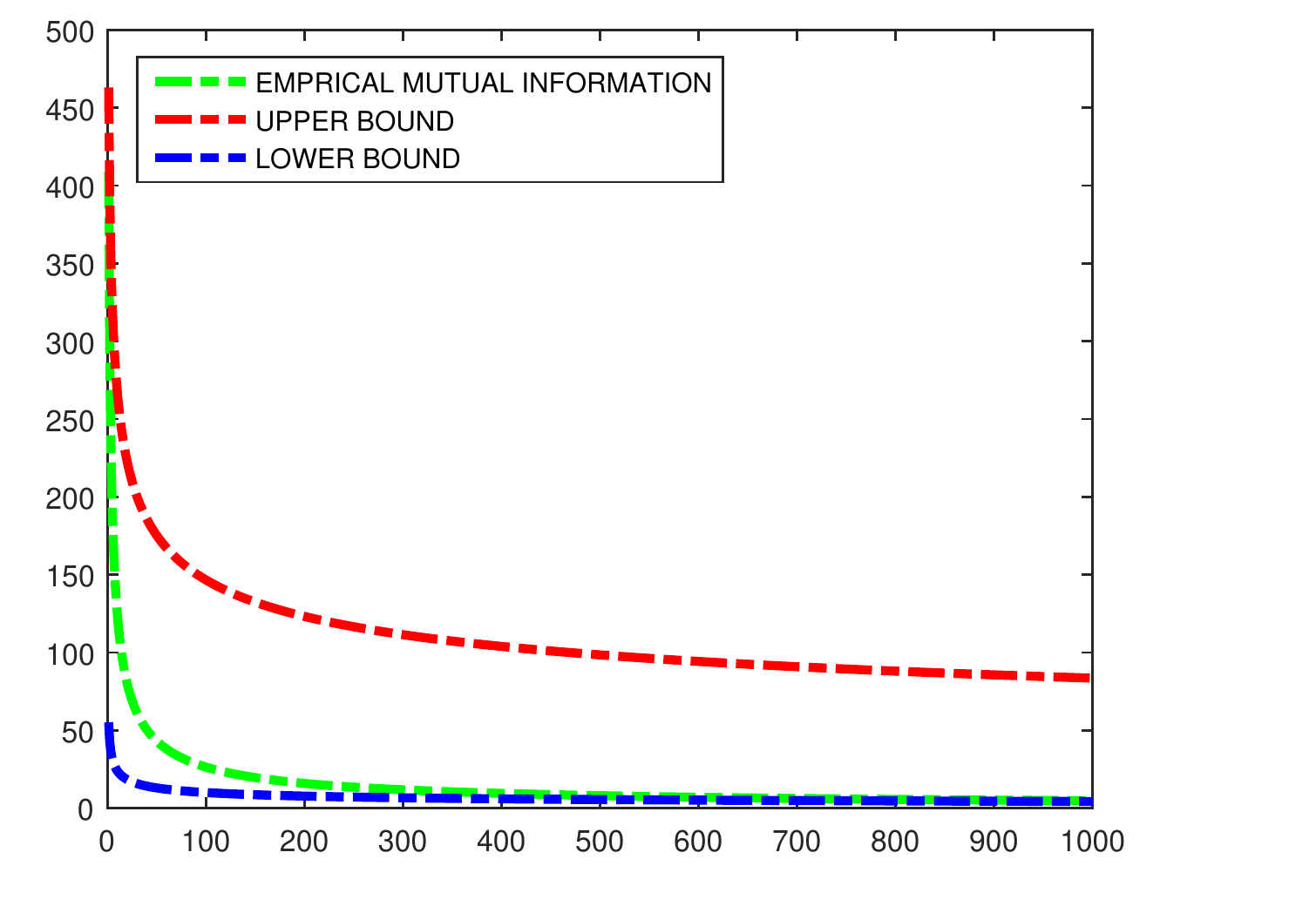}

\includegraphics[height=4cm, width=4cm]{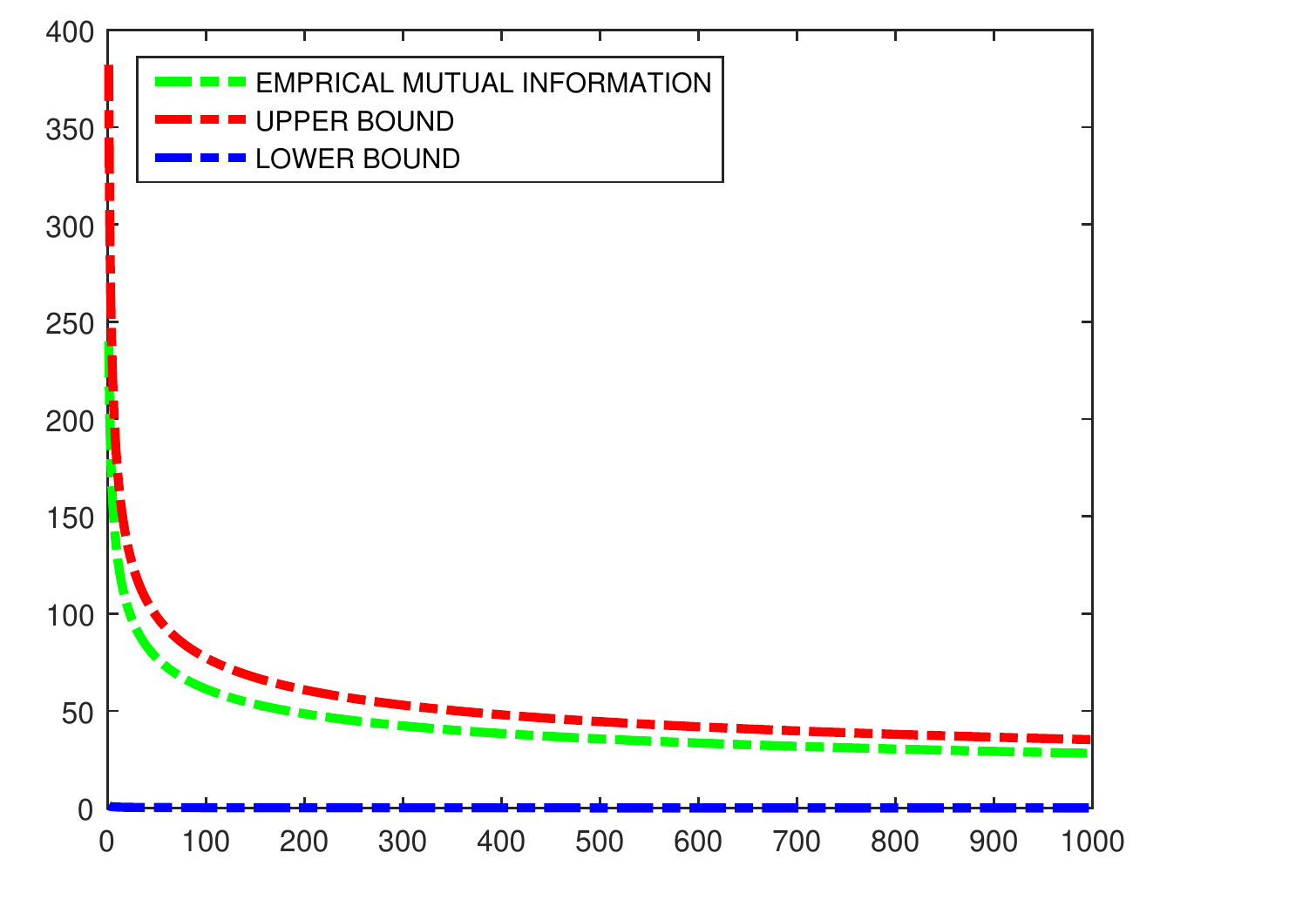}
\includegraphics[height=4cm, width=4cm]{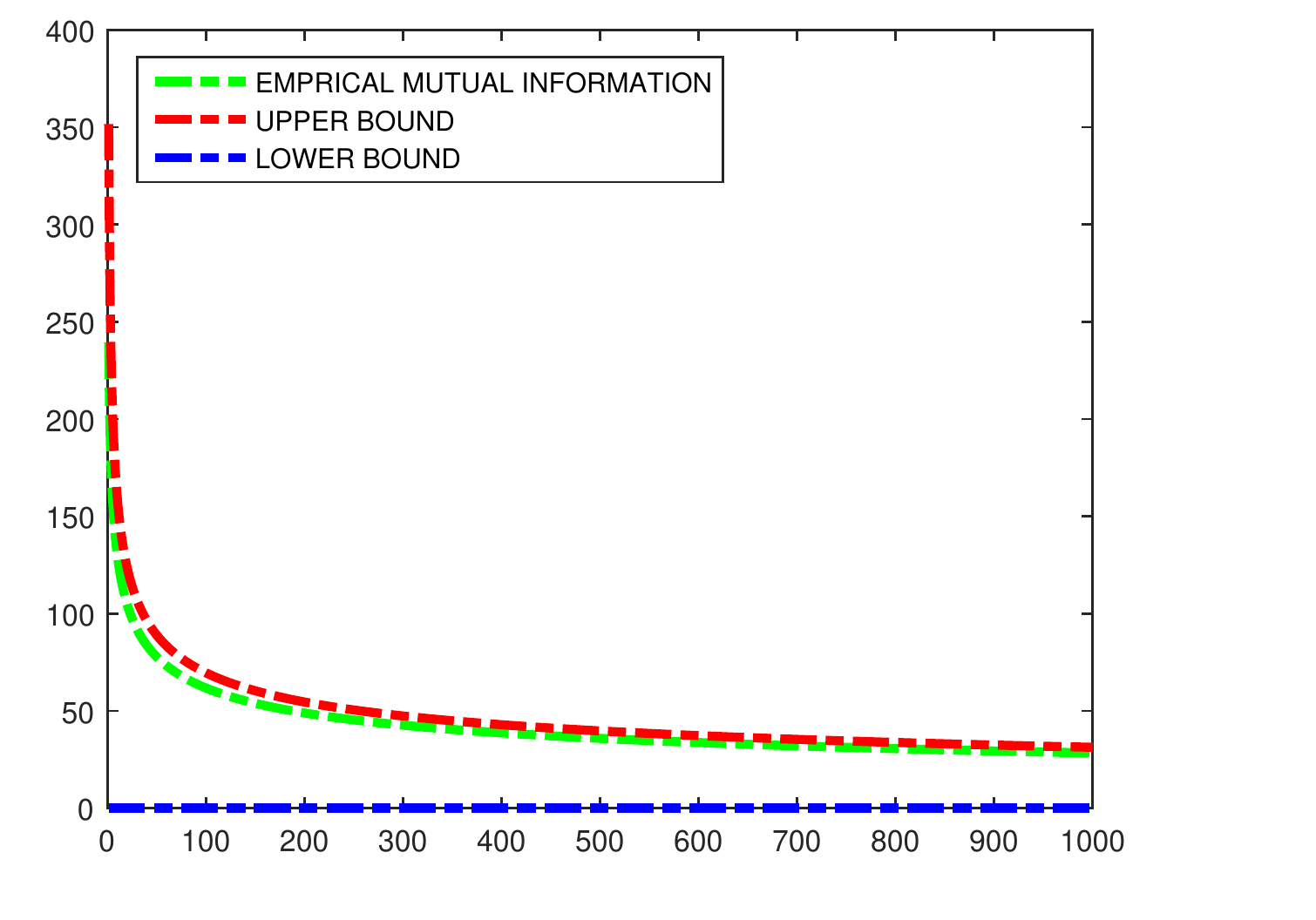}
\includegraphics[height=4cm, width=4cm]{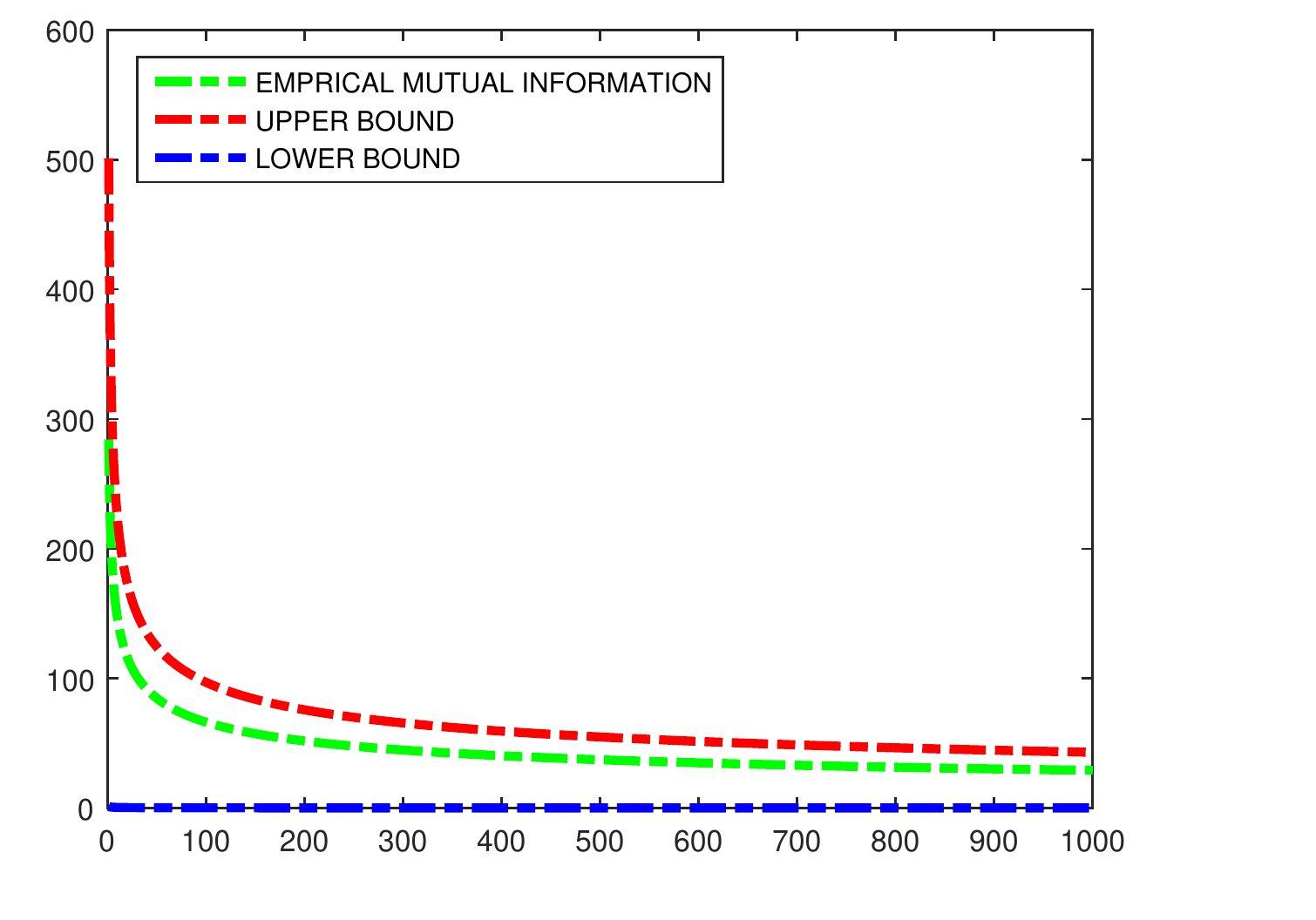}

\includegraphics[height=4cm, width=4cm]{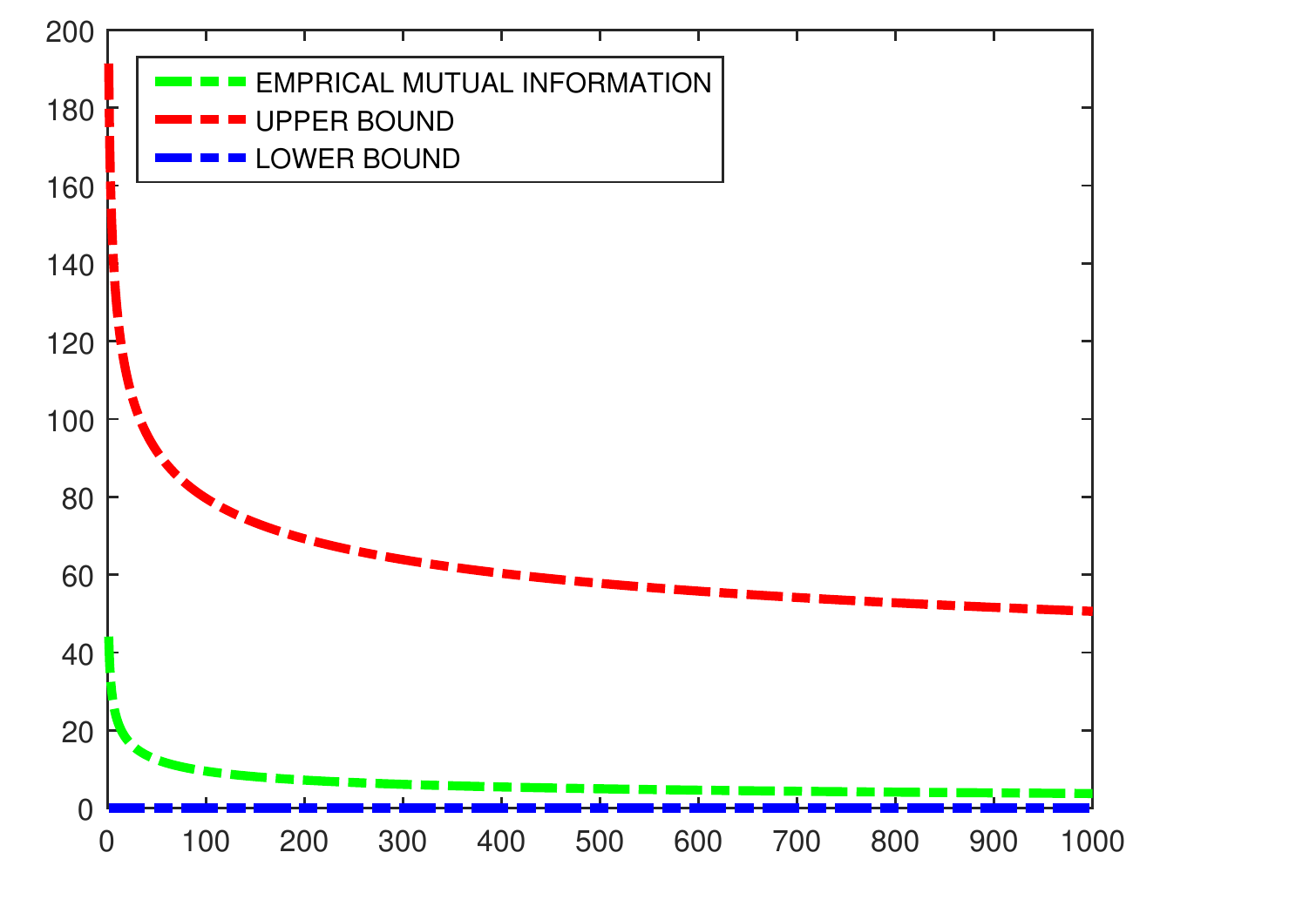}
\includegraphics[height=4cm, width=4cm]{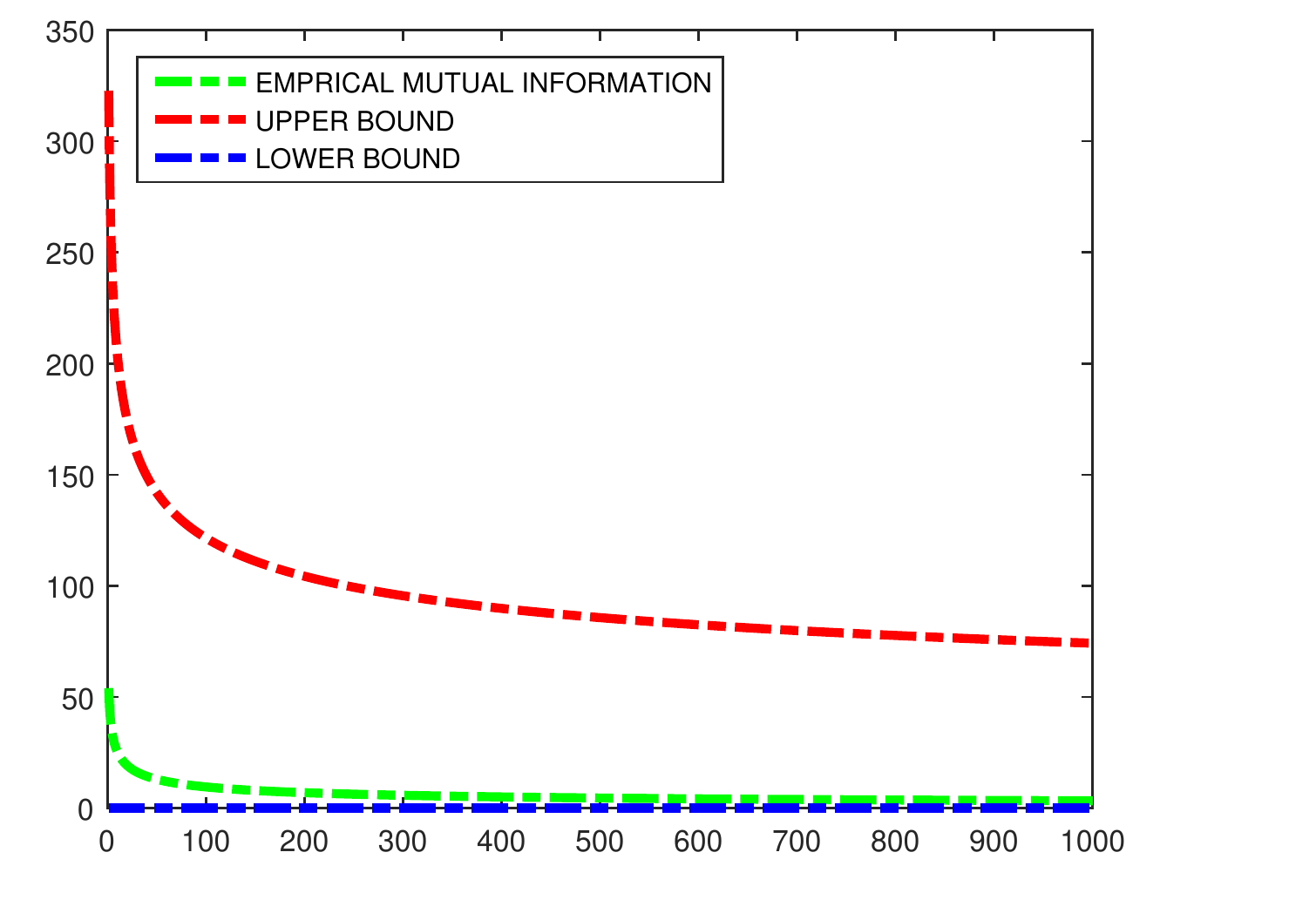}
\includegraphics[height=4cm, width=4cm]{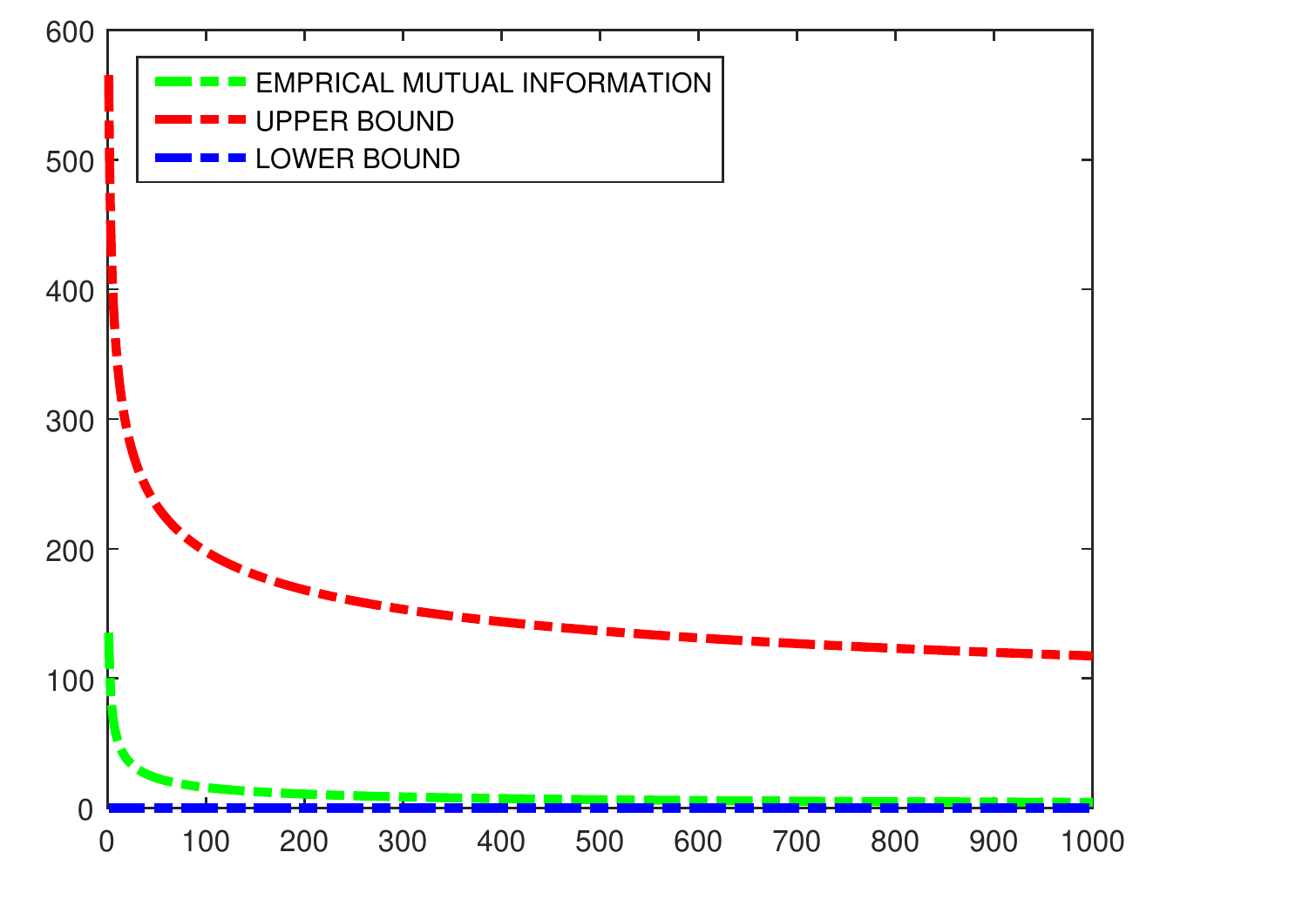}

\includegraphics[height=4cm, width=4cm]{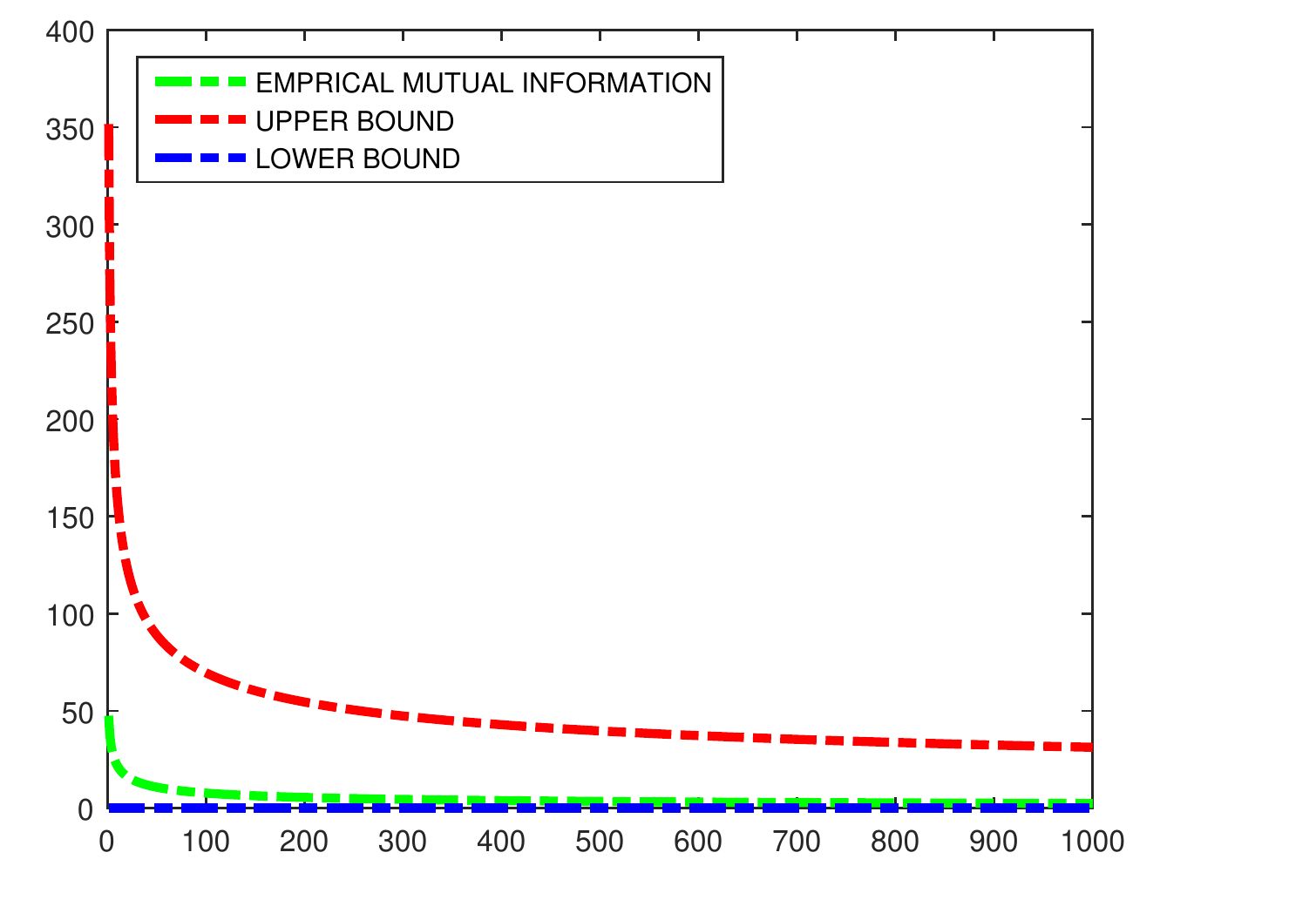}
\includegraphics[height=4cm, width=4cm]{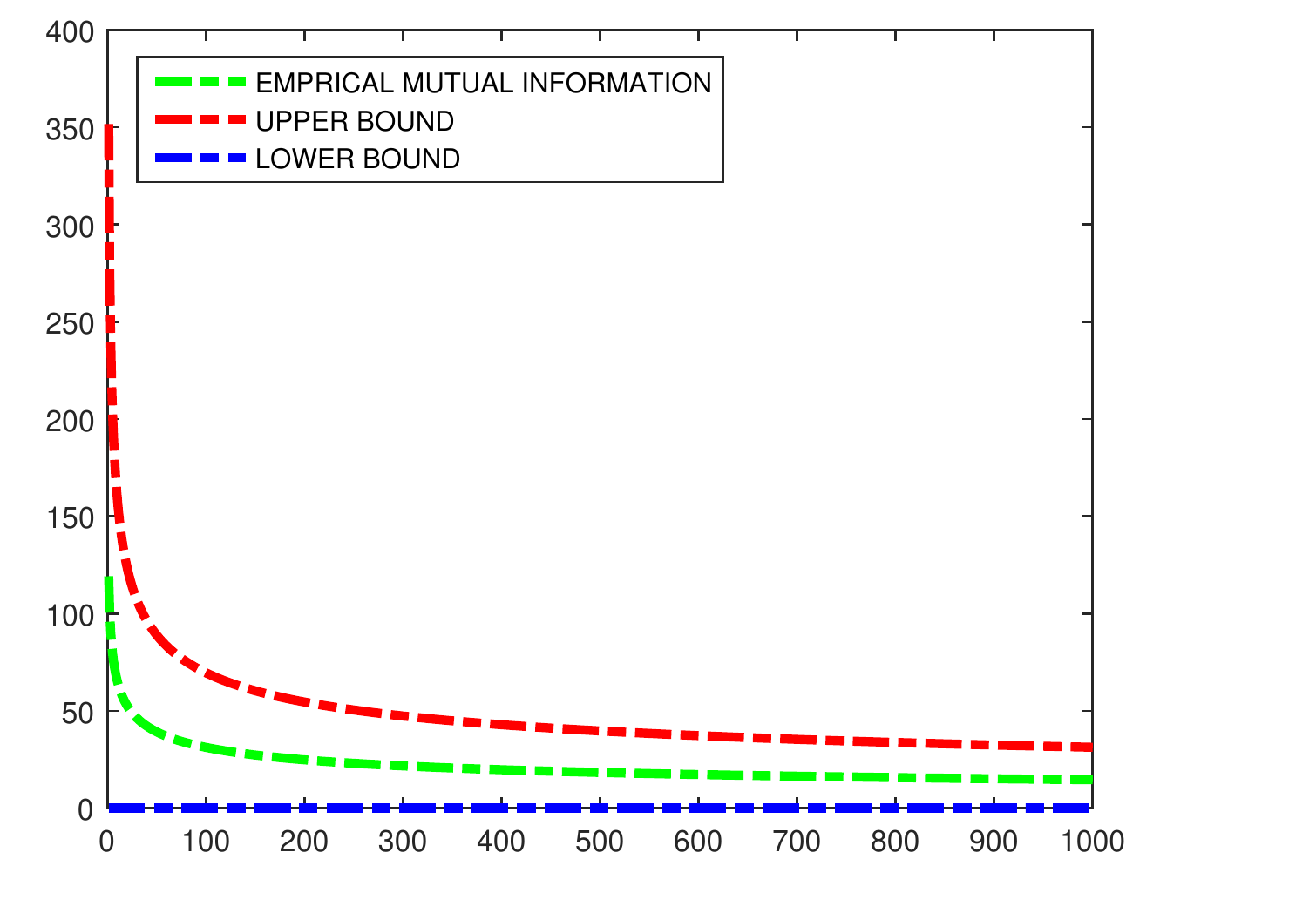}
\includegraphics[height=4cm, width=4cm]{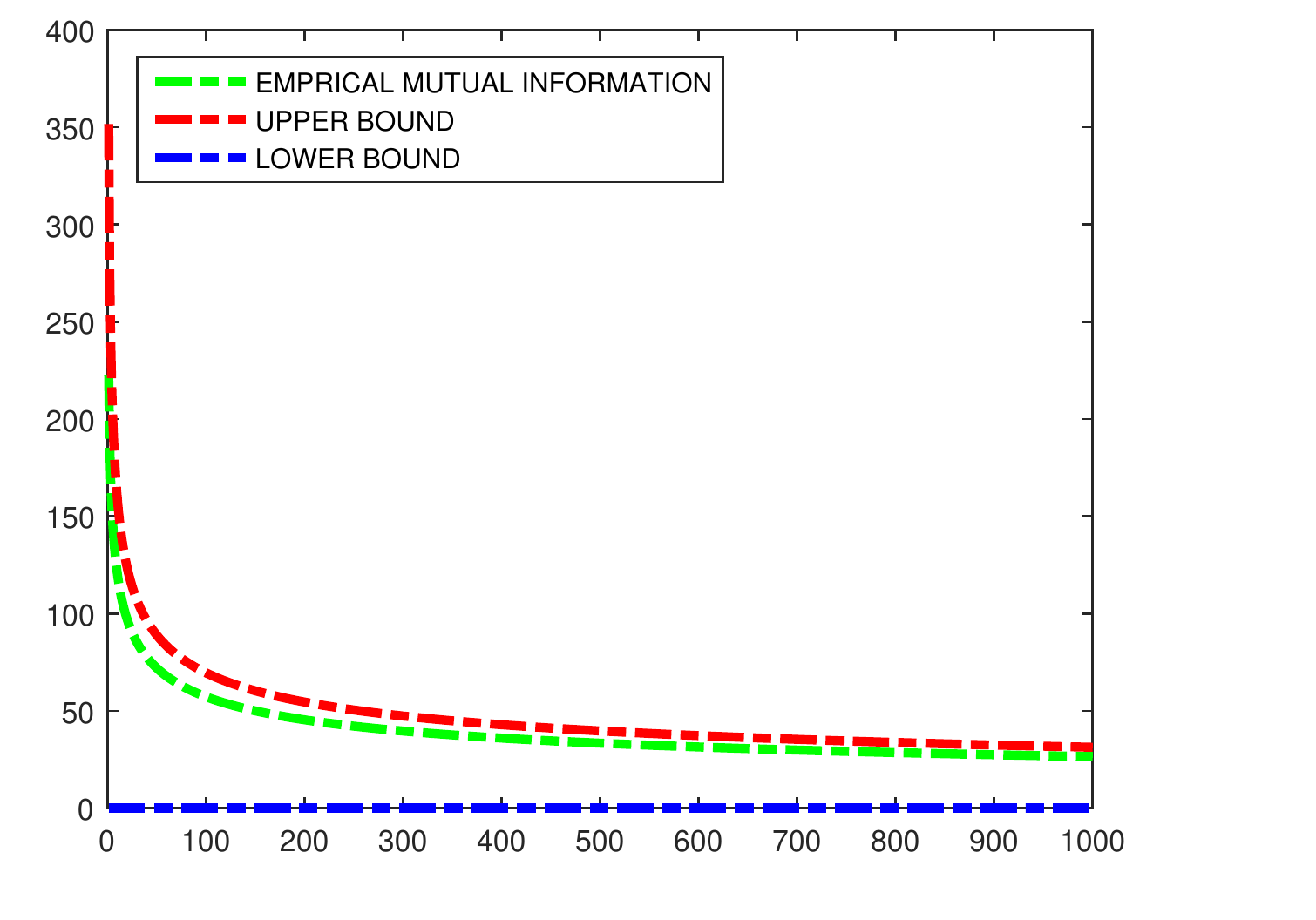}
\end{center}
\caption{{In green color, truncated $I_{q}(X(\mathbf{x}),X(\mathbf{y})),$ $q=1.5,~2.10,~2.25$   (top-row); $I_{q}(\chi_{10}^{2}(\mathbf{x}),\chi_{10}^{2}(\mathbf{y})),$  $q=2,~2.05,~2.10$  (second row);
$I_{q}(g(X(\mathbf{x})),g(X(\mathbf{y}))),$ $q=1.5,~1.75,~1.95$  (third row), and $I_{q}(g(\chi_{10}^{2}(\mathbf{x})),g(\chi_{10}^{2}(\mathbf{y}))),$ $q=1.75,~1.85,~1.95,$ for $g(x)=1_{\nu}(x)$ (fourth  row).
 The upper and lower bounds are represented in red and blue colors, respectively.  As before, $g(x)=1_{\nu}(x),$ $\nu=0.95$\label{fig:62bb}}}
\end{figure}

\section{Spatiotemporal case: functional approach}
\label{section:st-case}

In this section, we consider the extension of the above introduced concepts and elements in an infinite-dimensional framework. In this sense, a wider concept of  diversity is adopted for  functional systems characterized by separable non-countable families of infinite-dimensional random variables, and their measurable functions.

Specifically, the departure from independence of the components of a  functional  system, as displayed by  spatial white noise random fields,  is measured in terms of  diversity loss in the spatial functional sample paths. Equivalently,  diversity loss is induced here  by the strong interrelations displayed by the functional random components of such systems.

\subsection{Mutual information in an infinite-dimensional framework}

The formulation of mutual information as a measure for spatiotemporal structural complexity analysis can be addressed for the general class of Lancaster-Sarmanov random fields adopting the  infinite-dimensional spatial framework introduced in Angulo and Ruiz-Medina (2023).

Let $X=\{X_{\mathbf{x}}(\cdot ),\ \mathbf{x}\in \mathbb{R}^d \}$  be a zero-mean homogeneous and isotropic spatial functional random field on the separable Hilbert space $(H, <\cdot,\cdot>_H)$, mean-square-continuous w.r.t. the $H$ norm. In the following, we will assume that $H=L^{2}(\mathcal{T}),$ with $\mathcal{T}\subseteq\mathbb{R}_{+}$. For every $\mathbf{x},\mathbf{y}\in \mathbb{R}^d$,  $\left(X_{\mathbf{x}}(\cdot ),X_{\mathbf{y}}(\cdot)\right)^{T}$ is a random element in the separable Hilbert space   $\left(H^{2}, \left\langle \cdot,\cdot\right\rangle_{H^{2}}\right)$ of vector functions $\mathbf{f}=(f_{1},f_{2})^{T}$, with the inner product  given by
$\left\langle \mathbf{f},\mathbf{g}\right\rangle_{H^2}=\sum_{i=1}^{2}\left\langle f_{i}, g_{i}\right\rangle_{H},\ \forall \mathbf{f},\mathbf{g}\in H^{2}$.
Thus, for every $\mathbf{x},\mathbf{y}\in \mathbb{R}^d$,  we consider the measurable function $$\left(X_{\mathbf{x}}(\cdot ),X_{\mathbf{y}}(\cdot)\right)^{T}: \;(\Omega,\mathcal{A}, P)\longrightarrow \left( H^{2},\mathcal{B}(H^{2}),P_{\|\mathbf{x}-\mathbf{y}\|}(dh_{1},dh_{2})\right).$$
Let us denote by $\{P_{X_{\mathbf{x}}(\cdot)}(dh),\ \mathbf{x}\in \mathbb{R}^d\}$ the marginal infinite-dimensional probability  distributions,
with $P_{X_{\mathbf{x}}(\cdot)}(dh)=P(dh),$
  for every $\mathbf{x}\in \mathbb{R}^d.$ Let $L^{2}(H,P(dh))$ be the space of measurable functions $g:\ H \longrightarrow \mathbb{R}$ such that $\int_{H}|g(h)|^{2}P(dh)<\infty.$
Assume that there exists an orthonormal basis $\{\mathcal{B}_{k},\ k\geq 0\}$ of $L^{2}(H,P(dh))$  such that the Radon-Nikodym derivative of the bivariate infinite-dimensional probability distribution $P_{\|\mathbf{x}-\mathbf{y}\|}(dh_{1},dh_{2})$ can be written in terms of the corresponding marginals as (see, e.g., Ledoux  and  Talagrand, 1991), for $n,m\geq 1,$
\begin{eqnarray} &&p_{\|\mathbf{x}-\mathbf{y}\|,n,m}(h_{1},h_{2})= p(h_{1})p(h_{2})
\nonumber\\
&&\hspace*{-1cm}\times  \left[1+\sum_{k\geq 1}
\left[\int_{H}\left\langle X_{\mathbf{x}}(\cdot),\phi_{n}(\cdot )\right\rangle_{L^{2}(\mathcal{T})}
\left\langle X_{\mathbf{y}}(\cdot),\phi_{m}(\cdot )\right\rangle_{L^{2}(\mathcal{T})} P_{\|\mathbf{x}-\mathbf{y}\|}
(dX_{\mathbf{x}}(\cdot),dX_{\mathbf{y}}(\cdot))\right]^{k} \right.
\nonumber\\
&&\hspace*{5cm}\left.\times
\mathcal{B}_{k}(h_{1}) \mathcal{B}_{k}(h_{2})\right]\mu(dh_{1}, dh_{2}),\nonumber
\end{eqnarray}
\noindent for a given orthonormal basis  $\{ \phi_{n},\ n \geq 1\}$ of $H,$ with \begin{eqnarray} & & \gamma_{\|\mathbf{x}-\mathbf{y}\|}(\phi_{n})(\phi_{m}) =\mbox{Corr} (X_{\mathbf{x}}(\phi_{n}), X_{\mathbf{y}}(\phi_{m}))\nonumber\\ &&=\int_{H}\left\langle X_{\mathbf{x}}(\cdot),\phi_{n}(\cdot )\right\rangle_{L^{2}(\mathcal{T})}
\left\langle X_{\mathbf{y}}(\cdot),\phi_{m}(\cdot )\right\rangle_{L^{2}(\mathcal{T})} P_{\|\mathbf{x}-\mathbf{y}\|}
(dX_{\mathbf{x}}(\cdot),dX_{\mathbf{y}}(\cdot))\nonumber\end{eqnarray} \noindent  being the  spatial correlation operator applied to the elements $\phi_{n}$ and $\phi_{m}$ of the orthonormal basis $\{ \phi_{n},\ n \geq 1\},$
 for every $\mathbf{x},\mathbf{y}\in \mathbb{R}^d.$  Here,  $p(h)$ denotes  the Radon-Nikodym derivative of the absolutely continuous marginal infinite--dimensional probability measure $P(dh),$  with respect to the uniform probability measure $\mu (dh)$ in $H.$ The bivariate uniform measure    on $H^{2}$ is denoted as  $\mu(dh_{1}, dh_{2}).$

Under the above setting of conditions, the resulting class of spatial functional random fields defines the infinite--dimensional version  of Lancaster-Sarmanov random fields. In the  simulation study undertaken in the  next section, we analyze the asymptotic behavior of the Shannon entropy based mutual information between two random components of an element of this functional random field class. Specifically, from the infinite-dimensional formulation of Kullback-Leibler divergence established in
Angulo and Ruiz-Medina (2023), we consider the following version   of Shannon mutual information operator as a functional  counterpart  of equation (\ref{mie}): For $\mathbf{x},\mathbf{y}\in \mathbb{R}^d$, and $n,m\geq 1,$
\begin{eqnarray}&&\mathcal{S}_{\varrho }(\|\mathbf{x}-\mathbf{y}\|,n,m):=\mathcal{S}_{\varrho }(\|\mathbf{x}-\mathbf{y}\|)(\phi_{n})(\phi_{m})
\nonumber\\
 &&\int_{H^{2}}  P_{\|\mathbf{x}-\mathbf{y}\|,n,m}(dh_{1}, dh_{2})
 \ln\left(\frac{p_{\|\mathbf{x}-\mathbf{y}\|,n,m}(h_{1},h_{2})}{p(h_{1})p(h_{2})}\right)\nonumber\\
 &&=\int_{H^{2}}  p(h_{1})~p(h_{2})~\left[ 1+\sum_{k=1}^{\infty }[\gamma_{\|\mathbf{x}-\mathbf{y}\|}(\phi_{n})(\phi_{m})]^{k}\right.
 \nonumber\\ &&\left.\times \mathcal{B}_{k}(h_{1}) \mathcal{B}_{k}(h_{2})\right] \ln\left(1+\sum_{k=1}^{\infty }[\gamma_{
\|\mathbf{x}-\mathbf{y}\|}(\phi_{n})(\phi_{m})]^{k}~\mathcal{B}_{k}(h_{1}) \mathcal{B}_{k}(h_{2})\right) \mu(dh_{1}, dh_{2}).\nonumber\\
\end{eqnarray}

\subsection{Simulation study}
\label{stsimul}

Let $\varphi (u)$, $u\geq 0$, be a completely monotone function, and
suppose further that $\psi (u)$, $u\geq 0$, is a positive function with a
completely monotone derivative (such functions are also called Bernstein
functions). Consider the function
\begin{equation}
C(\left\Vert z\right\Vert ,\tau )=\frac{\sigma ^{2}}{[\psi (\tau ^{2})]^{d/2}}%
\varphi \left( \frac{\left\Vert z\right\Vert ^{2}}{\psi (\tau ^{2})}\right) ,%
\text{ }\sigma ^{2}\geq 0,\text{ }(z,\tau )\in \mathbb{R}^{d}\times \mathbb{R%
}.  \label{gneit}
\end{equation}
Hence $C$ is a  covariance function under Gneiting's criterion (see Gneiting, 2002).
Let now consider the special case of functions $\varphi $ and $\psi$ given by
\begin{eqnarray}
\varphi (u)&=&\frac{1}{(1+cu^{\gamma })^{\delta }},u>0,c>0,0<\gamma \leq 1,\delta >0\nonumber\\
\psi (u)&=&(1+au^{\alpha })^{\beta },\text{ }a>0,\text{ }0<\alpha \leq 1,\text{
}0<\beta \leq 1,\text{ }u\geq 0.\label{fcf}
\end{eqnarray}

Note that, functions $\varphi $ and $\psi $ can also be written as
\begin{eqnarray}
\varphi (u)&=&\frac{\mathcal{L}_{1}(u)}{u^{\gamma \delta}},\quad \mathcal{L}_{1}(u)=\frac{u^{\gamma \delta}}{(1+cu^{\gamma })^{\delta }}\nonumber\\
\psi(u) &=& \left[\frac{\mathcal{L}_{2}(u)}{u^{\alpha \beta}}\right]^{-1},\quad \mathcal{L}_{2}=\frac{u^{\alpha \beta}}{(1+au^{\alpha })^{\beta }},
\label{Taubth}
\end{eqnarray}
where $\mathcal{L}_{i},$ $i=1,2,$ represent positive continuous slowly varying function at infinity, satisfying
\begin{equation}
\lim_{T\to \infty}\frac{\mathcal{L}_{i}\left(T\|z\|\right)}{\mathcal{L}\left(T\right)}=1, \quad i=1,2,
\label{avf}
\end{equation}
\noindent for every $z\in \mathbb{R}^{d},$ $d\geq 1.$

Applying Tauberian Theorems (see Doukhan et al., 1996; see also Theorems 4 and 11 in Leonenko and  Olenko, 2014), their Fourier transforms  satisfy
 \begin{eqnarray}\widehat{\varphi}(\lambda )=\int_{\mathbb{R}^{d}}\exp\left(-i\left\langle \lambda , z\right\rangle \right) \varphi(\|z\|^{2})dz &\sim & c(1,2\gamma \delta ) \frac{\mathcal{L}_{1}\left( \frac{1}{\|\lambda \|}\right)}{\|\lambda\|^{d-2\gamma \delta }},\quad \mbox{as }\|\lambda \|\to 0\nonumber\\
 \widehat{\psi}(\omega )=\int_{\mathbb{R}}\exp\left(-i\left\langle \omega , t\right\rangle \right) \psi(|t|^{2})dt &\sim &  c(1,2\alpha \beta)\frac{\mathcal{L}_{2}\left( \frac{1}{|\omega|}\right)}{|\omega|^{1-2\alpha \beta}},\quad \mbox{as } |\omega | \to 0,\nonumber\\
  \label{asympsd}
\end{eqnarray}
where  $c\left( d,\theta \right) =\frac{\Gamma \left(\frac{d-\theta}{2}\right)}{\pi^{d/2}2^{\theta }\Gamma
(\theta/2)},$ with $\theta =2\gamma \delta ,$ and $\theta =2\alpha \beta,$ for $0< 2\gamma \delta <d,$ and  $0<2\alpha \beta <1.$

Figure \ref{fig:63}  displays color plots surfaces of the Gaussian spatiotemporal random field, generated with   covariance function in the Gneiting  class for $\alpha =0.3,$ $\beta = 0.7,$ $\gamma=0.2,$ and $\delta=0.35,$ at times $t=25, 50, 75,100.$ Shannon mutual information surfaces are evaluated  from the reconstruction formula $$\mathcal{K}_{\mathcal{S}_{\|\mathbf{x}-\mathbf{y}\| }}(t,s)=\sum_{n=1}^{M}\sum_{m=1}^{M}\mathcal{S}_{\varrho }(\|\mathbf{x}-\mathbf{y}\|)(\phi_{n})(\phi_{m})
\phi_{n}\otimes \phi_{m}(t,s)$$  \noindent  implemented  for $M=100.$ Figure \ref{fig:63}  displays   spatial cross--sections \linebreak
$\mathcal{K}_{\mathcal{S}_{\|\mathbf{x}_{i}-\mathbf{y}_{i}\| }}(t,s)$  corresponding to  spatial nodes $(i,1),$ $i=1,2,3,4,5$   (labeled  at the  x-axis), and $(i,1),$ $i=6,7,8,9,10$  (labeled at the y-axis), at $(t,s)\in [1,100]\times [1,100].$
 It can be observed the same power law in the decay of  $\mathcal{K}_{\mathcal{S}_{\|\mathbf{x}-\mathbf{y}\| }}(t,s)$  as $\|\mathbf{x}-\mathbf{y}\|\to \infty,$  for each fixed $(t,s)
\in [1,100]\times [1,100].$

\begin{figure}[htbp]
\begin{center}
 \includegraphics[height=8.7cm, width=10.7cm]{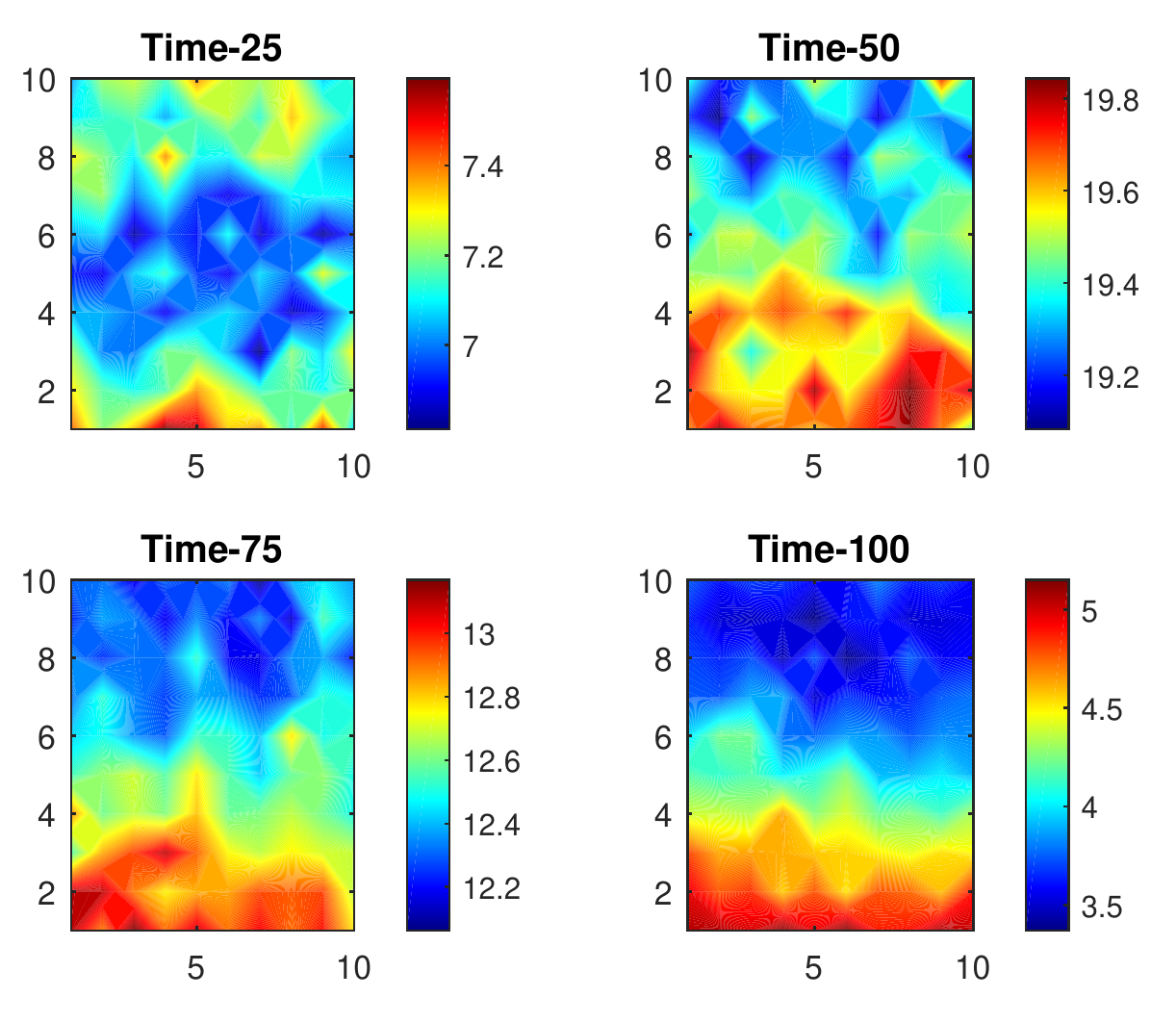}\hspace*{1cm}

\end{center}
\caption{{The surfaces values at times $t=25, 50, 75,100$  of the Gaussian spatiotemporal random field  generated with covariance function in the  Gneiting  class  are  displayed\label{fig:63}}}
\end{figure}

\begin{figure}[htbp]
\begin{center}
 \includegraphics[height=14cm, width=14cm]{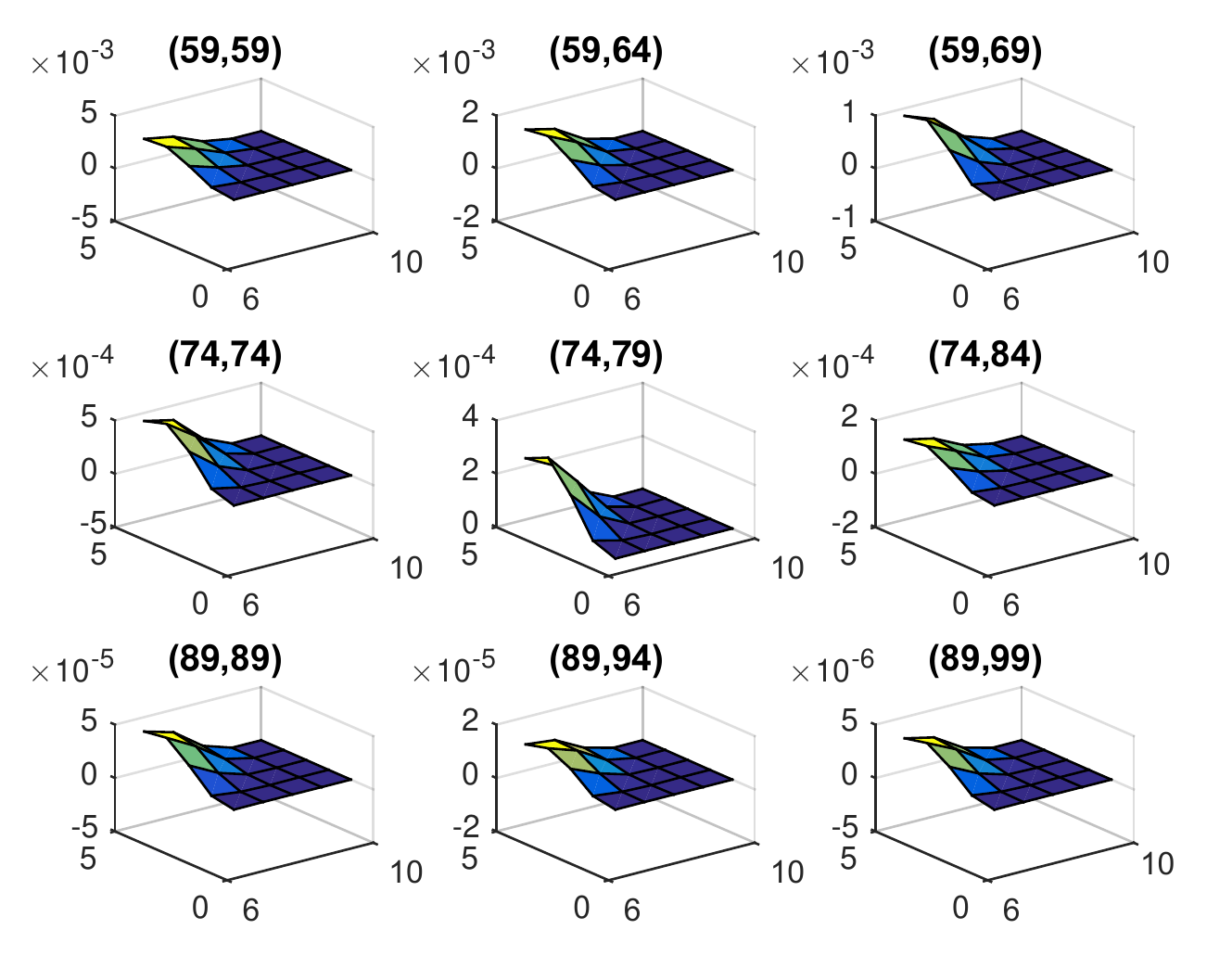}
\end{center}
\caption{{\scriptsize Shannon mutual information surfaces corresponding to crossing two random components respectively  located at spatial nodes $(i,1),$ $i=1,2,3,4,5$   (values labeled  at the x-axis), and $(i,1),$ $i=6,7,8,9,10$  (values labeled at the y-axis) of the $10\times 10$  spatial regular grid considered,     for crossing times in the set $\left[\{59,\dots, 99\}\right]^{2}.$  \label{fig:63b}}}
\end{figure}

\section{Conclusion}
\label{section:conclusion}

This paper focuses on the asymptotic mutual information based analysis of a class of spatial and spatiotemporal LRD Lancaster-Sarmanov random fields and their subordinated. Persistence  of memory in space is  characterized  in terms of the  LRD parameter, modelling the mutual information decay representing spatial structure dissipation at large scales.  Random field subordination  affects this decay rate    when the rank $m$ of the function $g$ involved in the subordination  is larger than one.   Hence,   large scale  aggregation  is lost up to order $m.$ Faster decay to zero is then observed in the corresponding asymptotic order modeling a faster spatial  structure dissipation from intermediate scales.  However,  when the rank is equal to one,  as illustrated here in the case of $g$ being the indicator function,  spatial structure dissipation occurs at the same rate, for  the original  and transformed  random variables.   The structural index provided by R\'enyi mutual information  reflects some different behaviors depending on the characteristics of the polynomial basis (in our simulation study, Hermite or Laguerre polynomials), as well as on the range analyzed for the $q$ deformation parameter. Particularly, this range induces strong changes at small spatial scales, but  the general shape of the curves reflecting asymptotic decay is invariant, and displays a power law involving the LRD  and the deformation   parameters.

In the spatiotemporal case, the class of Lancaster-Sarmanov random fields is introduced in a spatial functional framework. The simulation study undertaken shows a similar asymptotic behavior at spatial macroscale level, i.e., power law decay of the mutual information surfaces, which is accelerated at coarser temporal scales.  Thus, time varying asymptotic orders are obtained  characterizing the spatial diversity loss in a functional framework, under an increasing domain asymptotics.   Similar results will be derived, in a subsequent paper, regarding the asymptotic analysis of Shannon and R\'enyi mutual information measures, in an infinite-dimensional framework, in terms of   time-varying spatial  local complexity orders associated with a fixed domain asymptotics,  reflecting limiting behaviors at high resolution levels.

\section*{Acknowledgements}

This work has been supported in part by grants PGC2018-098860-B-I00 (J.M. Angulo), PGC2018-099549-B-I00 (M.D. Ruiz-Medina) and PID2021-128077NB-I00 (J.M. Angulo) funded by MCIN / AEI/10.13039/501100011033 / ERDF A way of making Europe, EU, and
grant CEX2020-001105-M funded by MCIN / AEI/10.13039/501100011033.

\end{document}